\documentclass{amsart}
\usepackage{graphicx} % Required for inserting images
\usepackage{amsmath}
\usepackage{amssymb}
\usepackage{amsthm}
\usepackage{mathrsfs}
\usepackage{amsfonts}
\usepackage{xcolor}
\usepackage[colorlinks = true, linkcolor = blue, citecolor=blue]{hyperref}
\usepackage[margin=1in]{geometry}
\usepackage{enumitem}
\usepackage{tikz}
\usepackage{comment}
\usepackage{stmaryrd}
\usepackage[OT2,T1]{fontenc}
\usepackage{cleveref}
\usepackage[normalem]{ulem}
\usepackage{makecell}

\newcommand{\Q}{\mathbb{Q}}

\newcommand{\F}{\mathbb{F}}

\renewcommand{\P}{\mathbb{P}}

\renewcommand{\O}{\mathcal{O}}
\newcommand{\Span}[1]{\langle #1 \rangle}
\newcommand{\sm}{\mathrm{sm}}

\renewcommand{\v}{\beta}
\newcommand{\w}{\gamma}
\newcommand{\z}{\delta}
\newcommand{\vvar}{B}
\newcommand{\wvar}{\Gamma}
\newcommand{\zvar}{\Delta}
\newcommand{\bv}{\boldsymbol{\v}}
\newcommand{\bw}{\boldsymbol{\w}}
\newcommand{\bz}{\boldsymbol{\z}}
\newcommand{\fvwz}{f_{\bv, \bw, \bz}}
\newcommand{\lmax}{\ell_{\max}}

\DeclareMathOperator{\GL}{GL}

\DeclareMathOperator{\Res}{Res}

\newtheorem{theoremAlph}{Theorem}

\newtheorem{theorem}{Theorem}[section]

\newtheorem*{conjecture*}{Conjecture}
\newtheorem{corollary}[theorem]{Corollary}
\newtheorem{lemma}[theorem]{Lemma}
\newtheorem{proposition}[theorem]{Proposition}
\theoremstyle{definition}
\newtheorem{remark}[theorem]{Remark}
\newtheorem{definition}[theorem]{Definition}
\newtheorem{example}[theorem]{Example}
\numberwithin{equation}{section}
\numberwithin{table}{section}

\newcommand{\qbound}[1]{{\color{cyan}$q > {#1}$}}
\newcommand{\qboundgeq}[1]{{\color{cyan}$q \geq {#1}$}}

\renewcommand{\qbound}[1]{{$q > {#1}$}}
\renewcommand{\qboundgeq}[1]{{$q \geq {#1}$}}

\title{Effective Bertini theorems and zeros of $p$-adic forms of degrees 7 and 11}

\author{Lea Beneish}
\address{Lea Beneish, Department of Mathematics, University of North Texas, Denton, TX, USA}
\email{lea.beneish@unt.edu}

\author{Christopher Keyes}
\address{Christopher Keyes, Department of Mathematics, King's College London, London, UK\newline \indent \hspace{1ex} and Heilbronn Institute for Mathematical Research, Bristol, UK}
\email{christopher.keyes@kcl.ac.uk}

\begin{document}

\begin{abstract}
    We establish an effective Bertini-type theorem for hypersurfaces $X_f \colon f = 0$ defined over a finite field $k$ for which $f$ has no linear factors over the algebraic closure $\overline{k}$. Given a line $L$ defined over $k$ and a nonreduced $\overline{k}$-point $x$ on $X_f \cap L$, we give an upper bound on the number of planes $P$ containing $L$ for which $X_f \cap P$ contains a line through $x$. Underlying this result is a factorization algorithm for bivariate polynomials originally due to Kaltofen, which we present with slightly relaxed hypotheses. Our primary application is to Artin's conjecture on $p$-adic forms of prime degree $d$: if $K/\Q_p$ is a finite extension with residue field isomorphic to $\F_q$ and $F \in K[x_0, \ldots, x_{d^2}]$ is homogeneous of degree $d$, the conjecture states $F$ has a nontrivial zero in $K$. We show this conjecture holds whenever \qbound{679} for $d=7$ and \qbound{7393} for $d=11$, improving upon a result of Wooley.
\end{abstract}

\maketitle

\section{Introduction}

In this work, we prove an effective Bertini-type theorem with a focus on ruling out certain factorizations in planar slices of a hypersurface in $\P^n$. Over a finite field $\F_q$ for $q$ sufficiently large, this ensures 
existence of a plane with the desired properties. We give an application of this to a conjecture of Artin concerning zeros of forms over $p$-adic fields.

\subsection{On Bertini theorems}
The classical Bertini theorem states that given a smooth projective variety $X \hookrightarrow \mathbb{P}^n$ over a field $k$, a generic hyperplane $H$ of $\mathbb{P}^n$ has the property that $H\cap X$ is smooth. In fact, there are several versions of this type of theorem where the adjective ``smooth'' is replaced with some other property of $X$, and a sufficiently general hyperplane intersected with $X$ is guaranteed to have the same property. If $k$ is infinite, the genericity of such a hyperplane is sufficient to guarantee the existence of one such hyperplane, but when $k = \F_q$ this need not be the case. One way to handle this is to replace the hyperplane $H$ by a hypersurface of higher degree: Poonen showed for $X$ a smooth quasiprojective subscheme of $\mathbb{P}^n$, a positive density of hypersurfaces have smooth intersections with $X$ \cite{Poonen_Bertini}.

In another direction, Kaltofen \cite{Kaltofen} gave an \textit{effective} Bertini theorem replacing smoothness with geometric irreducibility for planar slices: given a geometrically irreducible hypersurface $X \subset \P^n$ for $n \geq 3$, he gave an upper bound on number of planes $P \subset \P^n$ for which $X \cap P$ is not geometrically irreducible in terms of the degree of $X$. Over a finite field $\F_q$, this is essentially an effective version of the Hilbert Irreducibility Theorem. Cafure and Matera \cite{CafureMatera} modified this approach further to bound how often such a planar slice contains irreducible components of small degree. We give a modest improvement on these results and extend to reducible hypersurfaces (see Theorem \ref{thm:effBertini}). We also give a sharper result if one is interested only in the existence of a single such plane.

\begin{theoremAlph}[effective Bertini irreducibility for planes, existence form (see Theorem \ref{thm:effBertini_existence})]
\label{thm:effBertini_intro}
	Let $d \geq 2$, $n \geq 3$, and \qbound{2d^2}. Suppose $f \in \F_q[x_0,\ldots,x_n]$ is a degree $d$ form, irreducible over $\overline{\F_q}$, and $X_f \colon f = 0$ is the associated geometrically integral hypersurface in $\P^n$ defined over $\F_q$.
	\begin{enumerate}[label = (\roman*)]
		\item If \qbound{{\frac{d}{8}\left(3d^3 - 2d^2 - 3d + 2\right)}}, there exists a plane $P \subset \P^n$ defined over $\F_q$ such that $X_f \cap P$ is a geometrically integral plane curve of degree $d$, defined over $\F_q$.
		\item Fix an integer $0 < D < d$. If  
		\[q > \frac{dD}{8}\left(-D^3 + 4dD^2 - 6D^2 + 12dD - 11D + 8d - 6\right),\] 
		there exists a plane $P \subset \P^n$ defined over $\F_q$ such that $X_f \cap P$ is a plane curve of degree $d$, defined over $\F_q$, with no irreducible components of degree at most $D$ defined over $\overline{\F_q}$.
	\end{enumerate}
\end{theoremAlph}

A key technical tool is a more flexible version of \cite[Theorem 5]{Kaltofen} and \cite[Theorem 3.3]{CafureMatera}, which we briefly describe geometrically. Given a geometrically integral hypersurface $X_f$ of degree $d$ and a line $L$ meeting $X_f$ transversely in at least one marked geometric point, among planes $P$ containing $L$, $X_f \cap P$ is generically a geometrically integral plane curve of degree $d$. In Theorem \ref{thm:effBertini_modified}, we parametrize the planes $P$ and describe polynomial conditions which must be satisfied if $X_f \cap P$ has components of small degree containing the marked point. We use Theorem \ref{thm:effBertini_modified} not only as an intermediate step in the proof of Theorem \ref{thm:effBertini_intro}, but also find an application to Artin's conjecture for $p$-adic forms of degrees 7 and 11.

\subsection{On Artin's conjecture}

Let $\mathbb{Q}_p$ denote the field of $p$-adic numbers and let $K/\mathbb{Q}_p$ be a finite extension with residue field $\mathbb{F}_q$. Artin conjectured that $K$ is $C_2$; that is, for all positive integers $d$ and all $n \geq d^2$, every homogeneous polynomial $F \in K[x_0, \ldots, x_n]$ of degree $d$ has a nontrivial zero in $K$. This conjecture was proven false by Terjanian, who found a counterexample with $d=4$ and $n=17$ over $K = \mathbb{Q}_2$ \cite{Terjanian1966}; infinitely many counterexamples have since been shown to exist \cite{LewisMontgomery}.

However, weaker versions of the conjecture are known to hold. For a fixed degree $d$, we say that $K$ is $C_2(d)$ if for all $n \geq d^2$, every degree $d$ form $F \in K[x_0, \ldots, x_n]$ has a nontrivial zero in $K$. All $p$-adic fields $K$ are $C_2(2)$ and $C_2(3)$, due to Hasse and Lewis, respectively \cite{Lewis1952}. Using techniques from logic, Ax and Kochen showed that $p$-adic fields are \textit{asymptotically} $C_2(d)$, in the sense that for a fixed degree $d$, $\Q_p$ is $C_2(d)$ for all but finitely many primes $p$ \cite{AxKochen1965}. Moreover, a finite extension $K/\Q_p$ is $C_2(d)$ whenever $p \gg_{d, [K:\Q_p]} 0$. In fact, all of the known counterexamples to Artin's conjecture involve composite degrees $d$, divisible by $q-1$, leaving open the following conjecture.

\begin{conjecture*}[Artin's conjecture for prime degree $p$-adic forms]
\label{conj:artin}
    Let $K/\mathbb{Q}_p$ be a finite extension and $d > 3$ prime. For all $n \geq d^2$, every degree $d$ form $F \in K[x_0, \ldots x_n]$ has a nontrivial zero in $K$, i.e.\ $K$ is $C_2(d)$.
\end{conjecture*}

In degrees $d=7$ and $d=11$, we improve the best known effective lower bound on $q$ for which Artin's conjecture holds.

\begin{theoremAlph}\label{thm:deg711}
    Let $K/\Q_p$ be a $p$-adic field with residue field $\F_q$. 
    \begin{enumerate}[label = (\roman*)]
        \item If \qbound{679} then every degree $7$ form $F \in K[x_0, \ldots x_n]$ has a nontrivial zero in $K$, i.e.\ $K$ is $C_2(7)$.
        \item If \qbound{7393} then every degree $11$ form $F \in K[x_0, \ldots x_n]$ has a nontrivial zero in $K$, i.e.\ $K$ is $C_2(11)$.
    \end{enumerate}
\end{theoremAlph}

To establish Theorem \ref{thm:deg711}, we follow a $p$-adic minimization approach due to Birch and Lewis \cite{BirchLewis1959} and later refined by Laxton and Lewis \cite{LaxtonLewis1965} to give effective bounds on $q$ for which Artin's conjecture holds in degrees $d = 5, 7, 11$. They proposed to reduce to the case of so-called \textit{reduced} forms $F$, then use an effective version of the Lang--Weil bound for the number of $\F_q$-points on a variety to produce a nonsingular zero of $\overline{F}$, the image of $F$ under the residue map, when $q$ is sufficiently large. This then lifts to a $K$-zero of $F$ by Hensel's lemma. We recall the definition of reduced forms and basic facts about them in \S \ref{sec:reduced_forms}.

Refinements of this general strategy have shown that $K$ is $C_2(5)$ whenever \qbound{5} due to \cite{LeepYeomans1996_quintic_forms, HeathBrown2010_p-adic_zeros, Dumke2017} and to forthcoming work of the authors, and that $K$ is $C_2(7)$ (resp.\ $C_2(11)$) whenever \qbound{883} (resp.\ \qbound{8053}) due to \cite{Knapp2001, Wooley2008_7ic_11ic}. This strategy breaks down for prime degrees $d > 11$, since such $d$ can be written as the sum of composite numbers, which allows for the hypersurface given by $\overline{F}=0$ to have only totally nonreduced nonlinear components.

In degree 5, Leep and Yeomans \cite{LeepYeomans1996_quintic_forms} refined the $p$-adic minimization techniques of Birch and Lewis by exploiting the structure of \textit{singular} plane quintic curves. They argue that if $\overline{F}$ has no nonsingular zeros (which would then be liftable to $K$ by Hensel's lemma), then there exists a plane on which $\overline{F}$ restricts to a quintic curve with at least 3 singularities. The presence of these singularities reduces the $q$ threshold for which explicit point counting methods guarantee a liftable point. 

In contrast, in degree 7 (resp.\ 11), Wooley's bound \cite[Theorem 1]{Wooley2008_7ic_11ic} is obtained using the effective Bertini theorems \cite[Corollaries 3.2, 3.4]{CafureMatera} to find plane curves of degree 7 (resp.\ 11) with a geometrically integral component, then applying the Hasse--Weil bounds. In particular, Wooley does not exploit any singularities in the relevant degree 7 (resp.\ 11) curves, citing the absence of rigidity on the anatomy of a polynomial imposed by the relatively low degree of a quintic form.

Inspired by the work of Leep and Yeomans in the quintic case, we further exploit the structure of reduced forms of degrees 7 and 11 to either find nonsingular zeros of $\overline{F}$, or to show that $\overline{F}$ has two (necessarily singular) $\F_q$-zeros spanning a line $L$ meeting $\overline{X_F}$ transversely at a point. The latter is amenable to an application of our modified effective Bertini result, Theorem \ref{thm:effBertini_modified}, producing a plane curve with at least two singular $\F_q$-points and a geometrically irreducible component; these extra singularities lower the $q$ threshold for which we can find nonsingular points on the curve, and hence zeros of $F$ in $K$.

\subsection{Notation and organization} 

We reserve the following to have consistent meaning throughout:
\begin{itemize}
    \item $p$ denotes a prime;
    \item $q$ denotes a power of $p$;
    \item $K$ denotes $p$-adic field with residue field $\F_q$;
    \item $\overline{k}$ denotes the algebraic closure of a field $k$;  
    \item $F$ denotes a homogeneous form over $K$.
\end{itemize}
We will often write $f = \overline{F}$ for the image of $F$ under the residue map, to avoid confusion with the algebraic closure.

This paper is organized as follows. In Section \ref{sec:eff_bertini} we discuss Bertini theorems for irreducibility, stating two versions in Theorems \ref{thm:effBertini} and \ref{thm:effBertini_existence}, the latter of which specializes to Theorem \ref{thm:effBertini_intro}. In \S \ref{sec:factorization_alg} we describe a factorization algorithm for bivarate polynomials, which we employ in \S \ref{sec:genericity} to prove Theorem \ref{thm:effBertini_modified}. In Section \ref{sec:facts_degreed_Fq} we record some facts about degree 7 and 11 forms over $\F_q$, while in Section \ref{sec:reduced_forms_and_consequences} we recall the definition of reduced forms and prove some consequences in the degree 7 and 11 cases. Finally, we prove Theorem \ref{thm:deg711} in Section \ref{sec:proof of thmC}.

\subsection*{Acknowledgements}
This work began during a Research in Groups workshop at the International Centre for Mathematical Sciences in Edinburgh in June 2024. Further work was undertaken during Research in Pairs meetings at Banff International Research Station in December 2024 and the Centre International de Rencontres Mathématiques in Luminy in June 2025. The authors gratefully acknowledge the support of ICMS, BIRS, and CIRM for fostering a stimulating research environment, instrumental to the development of this work. The authors would also like to thank Tim Browning, Roger Heath-Brown, David Leep, Bjorn Poonen, and Trevor Wooley for thoughtful comments on an earlier draft.

LB is grateful for the support of the U.S. National Science Foundation (DMS-2418835) and the Simons Foundation (MPS-TSM–00007992). CK was supported by the Additional Funding Programme for Mathematical Sciences, delivered by EPSRC (EP/V521917/1) and the Heilbronn Institute for Mathematical Research, and an AMS-Simons Travel Grant.

\section{Effective Bertini theorems}
\label{sec:eff_bertini}

Let $k$ be an arbitrary field, $n \geq 3$, and suppose $f \in k[x_0,\ldots,x_n]$ is a degree $d$ form defined over $k$, irreducible over $\overline{k}$. Recall that generically, the intersection of the projective hypersurface $X_f \colon f = 0$ with a hyperplane remains geometrically irreducible; this is Bertini's theorem for irreducibility. However, when $k = \F_q$, this does not suffice to guarantee the existence of even one such hyperplane. 

By an \textit{effective} Bertini theorem, we refer to a result which bounds how often the non-generic behavior occurs. When $k = \F_q$, the approach can be adapted to give a lower bound on $q$ for which the generic behavior is guaranteed to occur at least once. We will be particularly interested in intersecting $X_f$ with a \textit{plane}, rather than a hyperplane, and find it convenient to work in affine coordinates.

Dehomogenizing by setting $x_0 = 1$, consider the family of bivariate polynomials given by
\[f_{\bv, \bw, \bz} = f(1, X + \v_1, \w_2X + \z_2 Y + \v_2, \ldots, \v_1, \w_n X + \z_n Y + \v_n) \in k[X,Y]\]
for
\[(\bv, \bw, \bz) = (\v_1, \ldots, \v_n, \w_2, \ldots, \w_n, \z_2, \ldots, \z_n) \in k^{3n-2}.\]
One may view $f_{\bv, \bw, \bz}$ as cutting out an affine plane curve, obtained by slicing the $x_0\neq 0$ affine patch of $X_f$ by the plane determined by $(\bv, \bw, \bz)$. If the original $f(x_0, \ldots, x_n)$ is absolutely irreducible, then generically so is $\fvwz$.

More generally, if $f$ factors as $f = \prod_{1 \leq i \leq t} f_i^{e_i}$ for $f_i$ distinct and irreducible and $e_i \geq 1$, we say $f$ has \textit{factorization type} $(d_1^{e_i},\ldots,d_t^{e_t})$. If $x_0 \nmid f$, then generically $\fvwz$ shares the same factorization type as $f$. For brevity, if $e_i = 1$, we omit it from the type notation. Note that the factorization type is not invariant under base extension; for the most part we will be interested in the factorization type over $\overline{k}$. 

For $f$ irreducible over $\overline{k}$, % (i.e.\ $t=1$), 
Kaltofen used an algorithm for testing irreducibility of a bivariate polynomial to deduce polynomial conditions on $(\bv,\bw,\bz)$ whose vanishing is implied by reducibility of $\fvwz$ \cite{Kaltofen}. Cafure and Matera modified this approach to detect how often $\fvwz$ has factors of small degree \cite{CafureMatera}. We offer an extension of their results in Theorem \ref{thm:effBertini} below, which specializes to a mild strengthening of \cite[Corollaries 3.2, 3.4]{CafureMatera} in the $t=1$ case.

\begin{theorem}[effective Bertini irreducibility for planes]
\label{thm:effBertini}
	Let $d \geq 2$ and $n \geq 3$. Suppose $f \in \F_q[x_0,\ldots,x_n]$ is a degree $d$ form with factorization type $(d_1^{e_1},\ldots,d_t^{e_t})$ over $\overline{\F_q}$ such that $x_0 \nmid f$.
	\begin{enumerate}[label = (\roman*)]
		\item For all but
		\[\left(2\left(\sum_{1 \leq i \leq t} d_i \right)^2 + \sum_{1 \leq i \leq t} \frac{d_i}{8}\left(3d_i^3 - 2d_i^2 - 3d_i + 2\right) \right)q^{3n-3}\]
		tuples $(\bv,\bw,\bz) \in \F_q^{3n-2}$, $\fvwz(X,Y)$ has factorization type $(d_1^{e_1},\ldots,d_t^{e_t})$ over $\overline{\F_q}$.
		
		\item Fix $0 \leq D_i < d_i$ for $1 \leq i \leq t$. For all but
		\[\left(2\left(\sum_{1 \leq i \leq t} d_i \right)^2 + \sum_{1 \leq i \leq t} \frac{d_iD_i}{8}\left(-D_i^3 + 4d_iD_i^2 - 6D_i^2 + 12d_iD_i - 11D_i + 8d_i - 6\right) \right)q^{3n-3}\]
		tuples $(\bv,\bw,\bz) \in \F_q^{3n-2}$, $(f_i)_{\bv,\bw,\bz}(X,Y)$ has no nonconstant factor of degree at most $D_i$ over $\overline{\F_q}$ for all $i$.
	\end{enumerate}
\end{theorem}

A consequence of Theorem \ref{thm:effBertini} is an effective bound on $q$ for which the existence of at least one $(\bv,\bw,\bz)$ for which $\fvwz$ has the same factorization type as $f$ is guaranteed (and similarly for $f$ lacking factors of small degree). We prove a stronger existence result  directly, under an additional hypothesis on $q$.

\begin{theorem}[effective Bertini irreducibility for planes, existence form]
\label{thm:effBertini_existence}
	Let $d \geq 2$, $n \geq 3$, and \qbound{2d^2}. Suppose $f \in \F_q[x_0,\ldots,x_n]$ is a degree $d$ form with factorization type $(d_1^{e_1},\ldots,d_t^{e_t})$ over $\overline{\F_q}$ such that $x_0 \nmid f$.
	
	\begin{enumerate}[label = (\roman*)]
		\item If 
		\[q > \sum_{1 \leq i \leq t} \frac{d_i}{8}\left(3d_i^3 - 2d_i^2 - 3d_i + 2\right),\]
		there exists $(\bv,\bw,\bz) \in \F_q^{3n-2}$ such that $\fvwz$ has factorization type $(d_1^{e_1},\ldots,d_t^{e_t})$ over $\overline{\F_q}$.
		
		\item Fix $0 \leq D_i < d_i$ for $1 \leq i \leq t$. If
        \[q > \sum_{1 \leq i \leq t} \frac{d_iD_i}{8}\left(-D_i^3 + 4d_iD_i^2 - 6D_i^2 + 12d_iD_i - 11D_i + 8d_i - 6\right),\]
        there exists $(\bv,\bw,\bz) \in \F_q^{3n-2}$ such that $(f_i)_{\bv,\bw,\bz}$ has no nonconstant factor of degree at most $D_i$ over $\overline{\F_q}$ for all $i$.
	\end{enumerate}
\end{theorem}

Before proving Theorems \ref{thm:effBertini} and \ref{thm:effBertini_existence}, we establish some intermediates. In \S \ref{sec:factorization_alg} we adapt Kaltofen's bivariate factorization algorithm, relaxing what amounts to a requirement that $\fvwz(X,0)$ is squarefree. In \S \ref{sec:genericity}, we use this to give a polynomial on $\bz$ which vanishes when $\fvwz(X,Y)$ has small factors through a fixed (geometric) point. The proofs of Theorems \ref{thm:effBertini} and \ref{thm:effBertini_existence} are then given in \S \ref{subsec:proofs_effBertini}. 

We conclude with some examples illustrating Theorems \ref{thm:effBertini} and \ref{thm:effBertini_existence} in the degree 7 case.

\begin{example}
\label{ex:deg7_irr}
    Let $d=7$, $n \geq 3$, and assume $f(x_0, \ldots, x_n)$ is irreducible over $\overline{\F_q}$. By Theorem \ref{thm:effBertini}(i), $\fvwz(X,Y)$ is irreducible over $\overline{\F_q}$ for all but $896q^{3n-3}$ choices of $(\bv,\bw,\bz) \in \F_q^{3n-2}$. 
    
    If we merely need existence of at least one such $(\bv,\bw,\bz)$, we see that \qbound{896} suffices. However, we can do slightly better using Theorem \ref{thm:effBertini_existence}(i), which states if \qbound{798}, there exists at least one $(\bv,\bw,\bz) \in \F_q^{3n-2}$ such that $\fvwz$ is irreducible over $\overline{\F_q}$.
\end{example}

\begin{example}
\label{ex:deg7_nolinear}
	Let $d=7$, $n \geq 3$, and assume $f(x_0, \ldots, x_n)$ is irreducible over $\overline{\F_q}$, as in Example \ref{ex:deg7_irr}. Suppose instead that we only require that $\fvwz(X,Y)$ does not have a \textit{linear} factor over $\overline{\F_q}$. Then applying Theorem \ref{thm:effBertini}(ii) with $D=1$, for all but $224q^{3n-3}$ choices of $(\bv,\bw,\bz) \in \F_q^{3n-2}$, we have $\fvwz$ has no linear factors. Using Theorem \ref{thm:effBertini_existence}(ii), if \qbound{126} then there exists at least one such $(\bv,\bw,\bz)$.
\end{example}

\begin{example}
\label{ex:deg7_reducible}
	Let $d=7$, $n \geq 3$, and suppose $f(x_0,\ldots,x_n)$ has factorization type $(d_1^{e_1},\ldots,d_t^{e_t})$ over $\overline{\F_q}$. Let $N_{(d_i^{e_i})}$ count the number of $\fvwz$ with the same type as $f$,
	\[N_{(d_i^{e_i})} = \#\left\{(\bv,\bw,\bz) \in \F_q^{3n-2} : \fvwz \text{ has type } (d_1^{e_1},\ldots,d_t^{e_t}) \right\}.\]
    We can use Theorem \ref{thm:effBertini}(i) to give an upper bound for $\# \F_q^{3n-2} - N_{(d_i^{e_i})}$, and Theorem \ref{thm:effBertini_existence}(i) to give a lower bound on $q$ for which at least one $\fvwz$ exists with type $(d_1^{e_1},\ldots,d_t^{e_t})$. This is illustrated in Table \ref{tab:deg7_bertini_examples} for types not containing a linear factor defined over $\overline{\F_q}$.
	
\begin{table}[ht]
\caption{Illustrating Theorems \ref{thm:effBertini}(i) and \ref{thm:effBertini_existence}(i) in degree 7}
\label{tab:deg7_bertini_examples}
\begin{tabular}{|c|c|c|}
	\hline \makecell{Factorization type of $f$\\$(d_1^{e_1},\ldots,d_t^{e_t})$} & \makecell{$\#\F_q^{3n-2} - N_{(d_i^{e_i})}$\\(Theorem \ref{thm:effBertini}(i))} & \makecell{$\exists \fvwz$ with type $(d_i^{e_i})$\\(Theorem \ref{thm:effBertini_existence}(i))} \\ \hline
	(7) & 		$\leq 896q^{3n-3}$ & \qbound{798} \\
	(5,2) & 	$\leq 296q^{3n-3}$ & \qbound{198} \\
	(4,3) & 	$\leq 194q^{3n-3}$ & \qbound{98} \\
	(3,2,2) & 	$\leq 125q^{3n-3}$ & \qbound{98} \\
	(3,$2^2$) & $\leq 122q^{3n-3}$& \qbound{98} \\ \hline 
\end{tabular}
\end{table}
\end{example}

\subsection{A factorization algorithm}
\label{sec:factorization_alg}

In this subsection only, suppose $f(X,Y) \in k[X,Y]$ is monic in $X$ and has total degree $d$ (not necessarily homogeneous). Fix a positive integer $D < d$. The factorization algorithms found in \cite{Kaltofen} and \cite[\S 3.1]{CafureMatera} underlying the proofs of effective Bertini theorems \cite[Corollaries 3.2, 3.3]{CafureMatera} request that $f(X,0)$ is squarefree. With slight modifications however, they still yield useful information when $f(X,0)$ is \textit{not squarefull}, i.e.\ when considered over $\overline{k}$, $f(X,0)$ has at least one simple root. We recall the setup for the algorithm; see also \cite[\S 1]{Kaltofen}, \cite[\S 3.1]{CafureMatera}.

Let $\alpha_0 \in \overline{k}$ be a root of $f(X,0)$ such that $\frac{\partial f}{\partial x}(\alpha_0, 0) \neq 0$. Viewing $f$ as a univariate polynomial in $X$ over $k(\alpha_0)[[Y]]$, since the partial derivative at $\alpha_0$ is a unit, we can lift $\alpha_0$ to $\widetilde{\alpha} \in k(\alpha_0)[[Y]]$ satisfying $f(\widetilde{\alpha}, Y) = 0$. Set $\lmax = dD$, our desired order of approximation. Truncating, we produce $\alpha \in k(\alpha_0)[Y]$ of $Y$-degree at most $\lmax$ satisfying
\begin{align*}
	f(\alpha, Y) &\equiv 0 \pmod{Y^{\lmax + 1}}.
\end{align*}
This may be accomplished explicitly via Newton's method, as in \cite{Kaltofen, CafureMatera}. For $\mu = 1, \ldots, D$ we write $\alpha^\mu$ as
\[\alpha^\mu \equiv \sum_{r=0}^{\lmax} a_{\mu, r} Y^r \pmod{Y^{\lmax+1}},\]
where the $a_{\mu, r} \in k(\alpha_0)$. 

We now look for a polynomial $g \in\overline{k}[X,Y]$ of degree $m \leq D$ for which $g \mid f$ and $g(\alpha_0,0) = 0$. Here we will use $\ell = dm$ as our order of approximation. For each such $m$ we examine whether
\begin{equation}\label{eq:powers_alpha}
	\alpha^m + \sum_{\mu=0}^{m-1} h_\mu(Y)\alpha^\mu \equiv 0 \pmod{Y^{\ell + 1}}
\end{equation}
has a solution for $h_\mu(Y) \in \overline{k}[Y]$ with $\deg h_\mu \leq m-\mu$. Writing
\[h_\mu(Y) = \sum_{\eta=0}^{m-\eta} u_{\mu,\eta}Y^\eta\]
for $u_{\mu,\eta} \in \overline{k}$, we have that \eqref{eq:powers_alpha} holds if and only if the linear system
\begin{equation}\label{eq:linear system}
	a_{m,r} + \sum_{\mu=0}^{m-1} \sum_{\eta=0}^{m-\mu} a_{\mu, r-\eta}u_{\mu,\eta} = 0 \text{ for } 0 \leq r \leq \ell
\end{equation}
has a solution in $\overline{k}$. Note that here we take the convention that $a_{\mu, r-\eta} = 0$ when $\eta > r$.

\begin{example}[$D=1$]
	As an explicit example, we describe \eqref{eq:powers_alpha} and \eqref{eq:linear system} in the special case of $D=1$. Here we take $m=1$, so we are looking to solve
	\[\alpha + h_0(Y) \equiv 0 \pmod{Y^{d+1}},\]
	for $h_0(Y) = u_{0,1}Y + u_{0,0} \in k(\alpha_0)[Y]$. We have
    \[\alpha + h_0(Y) = a_{1,0} + u_{0,0} + (a_{1,1} + u_{0,1})Y + \sum_{r=2}^{d} a_{1,r}Y^r,\]
    so a solution to \eqref{eq:powers_alpha} must have
	\begin{align*}
		a_{1,0} + u_{0,0} &= 0 \\
		a_{1,1}Y + u_{0,1}Y &= 0 \\
		a_{1,2}Y^2 &= 0 \\
		&\vdots \\
		a_{1,d}Y^{d} &= 0.
	\end{align*}
	Rearranging \eqref{eq:linear system} into a matrix equation over $\overline{k}$, this looks like
	\[\begin{pmatrix} 
		1 & 0 \\
		0 & 1 \\
		0 & 0 \\
		\vdots & \vdots \\
		0 & 0 
	\end{pmatrix}
	\begin{pmatrix}
		u_{0,0} \\ u_{0,1}
	\end{pmatrix} = 
	\begin{pmatrix}
		-a_{1,0} \\ -a_{1,1} \\ -a_{1,2} \\ \vdots \\ -a_{1,d}
	\end{pmatrix}.
	\]
	This system is insoluble for $d > 1$, unless $a_{1,r} = 0$ for $2 \leq r \leq d$, i.e.\ $\alpha \equiv a_{1,0} + a_{1,1}Y \pmod{Y^{d+1}}$ all along.
\end{example}

A solution to \eqref{eq:linear system} for some $m \leq D$ implies that $f$ has a nonconstant factor of degree at most $D$. We establish this below, along with a partial converse. This is closely related to \cite[Lemma 3.1]{CafureMatera}, which puts stronger hypotheses on $f$ and uses a different order of approximation.\footnote{In \cite[\S 3.1]{CafureMatera} $\lmax = 2dD$ and $\ell = 2dm$ were used.}

\begin{proposition}
\label{prop:linear_systems}
    Suppose $f(X,Y) \in k[X,Y]$ is a (not necessarily homogeneous) polynomial of total degree $d$ such that $f(X,0)$ is degree $d$ in $k[X]$ with a simple root $\alpha_0 \in \overline{k}$. Fix $D < d$.
    
    If the linear system \eqref{eq:linear system} has a solution in $k(\alpha_0)$ for $m \leq D$, then $f$ factors as $f = g(X,Y)h(X,Y)$ for nonconstant $g,h \in \overline{k}[X,Y]$ satisfying one of the following:
    \begin{enumerate}[label = (\alph*)]
    	\item $\deg g \leq D$ and $g(\alpha_0,0) = 0$;
    	\item $\deg g \leq D-1$.
    \end{enumerate}
    Conversely, if (a) above is satisfied, then the linear system \eqref{eq:linear system} has a solution in $k(\alpha_0)$ for some $m \leq D$.
\end{proposition}

\begin{proof}
    If \eqref{eq:linear system} has a solution in $k(\alpha_0)$ for $m \leq D$, then we have $\ell = dm$ and there exists $g(X,Y) \in k(\alpha_0)[X,Y]$ of total degree at most $m$, such that $g(\alpha,Y) \equiv 0 \pmod{Y^{\ell+1}}$. Consider $R(Y) = \mathrm{Res}_X(f,g) \in k(\alpha_0)[Y]$, which is a polynomial in $Y$ of degree at most $\ell$. Since $f(\alpha,Y) \equiv 0 \pmod{Y^{\ell+1}}$ as well, we have $R(Y) \equiv 0 \pmod{Y^{\ell+1}}$. This forces $R=0$, so $f$ and $g$ have a nontrivial common factor in $k(\alpha_0)[X,Y]$. Either $g\mid f$ and (a) is satisfied, or this common factor has degree strictly less than that of $g$, and hence less than $D$, so (b) is satisfied.
    
    For the partial converse, suppose $f = gh$ over $\overline{k}$ satisfies (a). Since $\alpha_0$ is a simple root of $f(X,0)$, we have $g$ must be defined over $k(\alpha_0)$; otherwise the conjugates of $g$ in $\overline{k}/k(\alpha_0)$ would also vanish at $(\alpha_0,0)$. Moreover, we have $h(\alpha_0,0) \neq 0$, so $h(\alpha,Y) \not\equiv 0 \pmod{Y^{\lmax+1}}$. Since $f(\alpha,Y) \equiv 0 \pmod{Y^{\lmax+1}}$, we have $g(\alpha,Y) \equiv 0 \pmod{Y^{\lmax+1}}$. The coefficients of $g$ thus produce a solution to \eqref{eq:linear system} in $k(\alpha_0)$ for $m = \deg g$.
\end{proof}

When $D=1$, Proposition \ref{prop:linear_systems} specializes to the statement that \eqref{eq:linear system} has a solution in $k(\alpha_0)$ for $m=1$ if and only if $f$ has a linear factor vanishing on $(\alpha_0,0)$. We can also extract an equivalence in the case $D= d-1$.

\begin{corollary}
\label{cor:linear_systems_reducibility}
    Suppose $f(X,Y) \in k[X,Y]$ is a (not necessarily homogeneous) polynomial of total degree $d$ such that $f(X,0)$ is degree $d$ in $k[X]$ with a simple root $\alpha_0 \in \overline{k}$. 

   The linear system \eqref{eq:linear system} has a solution in $k(\alpha_0)$ for some $m < d$ if and only if $f = gh$ for nonconstant $g,h \in k(\alpha_0)[X,Y]$.
\end{corollary}

% \begin{proof}
% 	The forward direction follows directly from Proposition \ref{prop:linear_systems}. For the reverse direction, assume without loss of generality that $g(\alpha_0,0) = 0$. Then apply Proposition \ref{prop:linear_systems} with $D = \deg g$.
% \end{proof}

\subsection{Genericity conditions}
\label{sec:genericity}

Our goal is to constrain non-generic behavior of $\fvwz$ to those $(\bv,\bw,\bz)$ that satisfy certain polynomial conditions. We reserve the lowercase $\v,\w,\z$ for elements of $k$, and use the uppercase $\vvar, \wvar, \zvar$ for indeterminates to describe these conditions. We give a flexible such result below; this is the key intermediate in proving Theorems \ref{thm:effBertini} and \ref{thm:effBertini_existence}. 

\begin{theorem}
\label{thm:effBertini_modified}
	Let $d \geq 2$ and $n \geq 3$. Suppose $f \in k[x_0,\ldots,x_n]$ is a degree $d$ form, irreducible over $\overline{k}$. Moreover, suppose there exist $(\bv, \bw) \in k^{2n-1}$ such that
	\[f(1,X + \v_1, \w_2X + \v_2, \ldots, \w_nX + \v_n) \in k[X]\]
   	has degree $d$ with a simple root $\alpha_0 \in \overline{k}$. Fix $0 \leq D < d$. 
   	
   	There exists $\Psi_D^{(\alpha_0)} \in k(\alpha_0)[\zvar_2, \ldots, \zvar_n]$ (depending on $(\bv, \bw)$) of degree
    \[\deg \Psi_D^{(\alpha_0)} \leq \frac{D}{8} \left(-D^3 + 4dD^2 - 6D^2 + 12dD - 11D + 8d - 6\right)\]
    such that whenever $\Psi_D^{(\alpha_0)}(\z_2,\ldots,\z_n) \neq 0$, $f_{\bv, \bw, \bz}(X,Y)$ has no nonconstant factor of degree at most $D$ over $\overline{k}$ vanishing on $(\alpha_0, 0)$.
    
    In particular, whenever $\Psi_{d-1}^{(\alpha_0)}(\z_2,\ldots,\z_n) \neq 0$, $\fvwz(X,Y)$ is irreducible over $\overline{k}$.
\end{theorem}

\begin{proof}
We follow similar lines as \cite[Theorem 5]{Kaltofen} and \cite[Theorem 3.3]{CafureMatera}. Let 
    \[\chi = \chi(X,Y,\zvar_2, \ldots, \zvar_n) = f(1,X +\v_1,\w_2X + \zvar_2Y +\v_2, \ldots,\w_nX + \zvar_nY +\v_n),\]
    considered in $\overline{k}[X,Y,\zvar_2,\ldots,\zvar_n]$. We have $\chi$ is irreducible by \cite[Lemma 5]{Kaltofen}.

    Now consider $\chi \in \overline{k}(\zvar_2, \ldots, \zvar_n)[X,Y]$. By Proposition \ref{prop:linear_systems}, the system \eqref{eq:linear system} has no solutions in the field $\overline{k}(\zvar_2, \ldots, \zvar_n)$ for $m\leq D$. Moreover, upon examination of the Newton iteration step \cite[Step N]{Kaltofen, CafureMatera} we observe the system \eqref{eq:linear system} has coefficients in $k(\alpha_0)[\zvar_2, \ldots, \zvar_n]$; the only denominators introduced are those involving $\frac{\partial \chi}{\partial X}(\alpha_0, 0) \in k(\alpha_0)$.

    Let $\Psi_D^{(\alpha_0)}$ be a maximal nonzero minor of the augmented matrix of \eqref{eq:linear system}. If $\bz \in k^{n-1}$ is chosen such that $\Psi_D^{(\alpha_0)}(\z_2, \ldots, \z_n) \neq 0$ then the system \eqref{eq:linear system} obtained by specializing $\zvar_i$ to $\z_i$ for $2 \leq i \leq n$ is inconsistent. Thus by Proposition \ref{prop:linear_systems}, $\chi(X,Y,\z_2,\ldots,\z_n) = f_{\bv, \bw, \bz}(X,Y)$ has no nonconstant factors of degree at most $D$ defined over $\overline{k}$ vanishing on $(\alpha_0, 0)$.

	To estimate the degree of $\Psi_{D}^{(\alpha_0)}$, the arguments of \cite[Theorem 3.3]{CafureMatera} (see also \cite[Theorem 5]{Kaltofen}) remain essentially unchanged, save for the updated value of $\lmax = dD$. We find
    \[\deg \Psi_D^{(\alpha_0)} \leq \sum_{j=0}^{(D+1)(D+2)/2-1} \lmax - j = \frac{D}{8} \left(-D^3 + 4dD^2 - 6D^2 + 12dD - 11D + 8d - 6\right),\]
    completing the proof.
\end{proof}

We will also need that $f(1,X+\v_1,\w_2X +\v_2,\ldots,\w_nX +\v_n)$ is generically squarefree of degree $d$, and in fact this can be controlled by another polynomial condition. This is \cite[Lemma 4]{Kaltofen}, modified to allow reducible $f$.

\begin{lemma}
\label{lem:upsilon}
	Let $d \geq 2$ and $n \geq 3$. Suppose $f \in k[x_0,\ldots,x_n]$ is a degree $d$ form with factorization type $(d_1,\ldots,d_t)$ over $\overline{k}$ such that $x_0 \nmid f$.
	
	There exists a nonzero polynomial $\Upsilon \in \overline{k}[\vvar_1, \ldots, \vvar_n, \wvar_2, \ldots, \wvar_n]$ with $\deg \Upsilon \leq 2d^2$ such that whenever $\Upsilon(\v_1, \ldots,\v_n,\w_2, \ldots,\w_n) \neq 0$ for $(\bv,\bw) \in \overline{k}^{2n-1}$, we have $f(1,X + \v_1, \w_2X + \v_2, \ldots, \w_nX + \v_n) \in \overline{k}[X]$ is squarefree of degree $d$.
\end{lemma}

\begin{proof}
	We essentially follow \cite[\S 3.2]{CafureMatera} and \cite[Lemmas 2,3,4]{Kaltofen}. Let 
	\[\phi = f(1, X + \vvar_1, \wvar_2X + Y_2 + \vvar_2, \ldots, \wvar_nX + Y_n + \vvar_n) \in k(\vvar_1, \ldots, \vvar_n, \wvar_2, \ldots, \wvar_n)[X,Y_2,\ldots, Y_n]\]
	and similarly define $\phi_i$ in terms of the irreducible factors $f_i$. By \cite[Lemma 2]{Kaltofen}, the $\phi_i$ are irreducible over the algebraic closure of $k(\vvar_1, \ldots, \vvar_n, \wvar_2, \ldots, \wvar_n)$. The same argument in \cite[Lemma 2]{Kaltofen} shows $\phi_i$ and $\phi_j$ share no common factors over the closure of $k(\vvar_1, \ldots, \vvar_n, \wvar_2, \ldots, \wvar_n)$, since otherwise so would the $f_i$. 
	
	Let $\lambda$ (resp.\ $\lambda_i$) denote the leading term of $\phi(X,0,\ldots,0)$ (resp.\ $\phi_i(X,0,\ldots,0)$) as a polynomial in $k(\vvar_1,\ldots,\vvar_n,\wvar_2,\ldots,\wvar_n)[X]$. We have $\lambda, \lambda_i \in k[\wvar_2,\ldots,\wvar_n]$ have total degree $d$ and $d_i$, respectively (the hypothesis $x_0 \nmid f$ ensures that $\phi(X,0,\ldots,0)$ has degree $d$ in $X$).
	
	Take $\psi = \phi/\lambda$ and $\psi_i = \phi_i/\lambda_i$, which are in $k(\vvar_1, \ldots, \vvar_n, \wvar_2, \ldots, \wvar_n)[X,Y_2,\ldots,Y_n]$, and now monic in $X$. We want to show the resultant $\Res_X\left(\psi(X,0,\ldots,0),\frac{\partial \psi(X,0,\ldots,0)}{\partial X}\right)$ is nonzero in $k(\vvar_1, \ldots, \vvar_n, \wvar_2, \ldots, \wvar_n)$, which \textit{generically} gives the desired squarefreeness. By basic properties of resultants, 
	\begin{align*}
		\Res_X&\Big(\psi(X,0,\ldots,0),\frac{\partial \psi(X,0,\ldots,0)}{\partial X}\Big)\\ 
        &= \prod_{1 \leq i \leq t} \Res_X\left(\psi_i(X,0,\ldots,0),\frac{\partial \psi_i(X,0,\ldots,0)}{\partial X}\right)\prod_{j \neq i} \Res_X\left(\psi_i(X,0,\ldots,0),\psi_j(X,0,\ldots,0)\right).
	\end{align*}
	By \cite[Lemma 3]{Kaltofen}, we have that $\Res_X\left(\psi_i(X,0,\ldots,0),\frac{\partial \psi_i(X,0,\ldots,0)}{\partial X}\right)$ is nonzero. Furthermore, following the strategy of the proof of \cite[Lemma 3]{Kaltofen}, we find that if $\Res_X\left(\psi_i(X,0,\ldots,0),\psi_j(X,0,\ldots,0)\right) \neq 0$, then there is a common nonconstant factor $g(X) \in k(\wvar_2,\ldots,\wvar_n)[X,\vvar_1,\ldots, \vvar_n]$ of $\psi_i(X,0,\ldots,0)$ and $\psi_j(X,0,\ldots,0)$. Making the substitution
	\[g\left(x_1-\vvar_1,\vvar_1,x_2 - \wvar_2(x_1-\vvar_1),\ldots,x_n - \wvar_n(x_1-\vvar_1)\right),\]
	and rehomogenizing, we obtain a nonconstant factor of both $f_i$ and $f_j$ over $k$, a contradiction. Thus we conclude the desired resultant is nonzero, from which it follows that $f(1,X + \v_1, \w_2X + \v_2, \ldots, \w_nX + \v_n)$ (and hence $\fvwz(X,0)$) is generically squarefree for $(\bv,\bw) \in \overline{k}^{2n-1}$.
	
	Take $\Upsilon$ to be the product of $\lambda$ and $\Res_X \left(\phi(X,0,\ldots,0),\frac{\partial \phi(X,0,\ldots,0)}{\partial X}\right)$. When $\Upsilon(\v_1,\ldots,\v_n,\w_2,\ldots,\w_n) \neq 0$, we have that
	\[f(1,X+\v_1,\w_2X +\v_2,\ldots,\w_nX +\v_n)\] 
	has nonzero leading term and discriminant, i.e.\ it is squarefree of degree $d$ in $\overline{k}[X]$. 
	
	It remains to compute the degree of $\Res_X \left(\phi(X,0,\ldots,0),\frac{\partial \phi(X,0,\ldots,0)}{\partial X}\right)$ in $k[\vvar_1,\ldots,\vvar_n,\wvar_2,\ldots,\wvar_n]$. As in \cite[Lemma 4]{Kaltofen}, this follows from the fact that the resultant is homogeneous of degree $2d-1$ in the coefficients of $\phi(X,0,\ldots,0)$ and $\frac{\partial \phi(X,0,\ldots,0)}{\partial X}$; these coefficients have total degree at most $d$ in $\vvar_1,\ldots,\vvar_n,\wvar_2,\ldots,\wvar_n$, making the total degree of the generic resultant at most $2d^2-d$. Adding in the degree of $\lambda$, we have $\deg \Upsilon \leq 2d^2$.
\end{proof}

\subsection{Proofs of Theorem \ref{thm:effBertini} and \ref{thm:effBertini_existence}}
\label{subsec:proofs_effBertini}

\begin{proof}[Proof of Theorem \ref{thm:effBertini}]	
	It suffices to prove the statement for $e_i = 1$ for $1 \leq i \leq t$, since $\fvwz = \prod_{1 \leq i \leq t}(f_i)^{e_i}_{\bv,\bw,\bz}$. 
	
	Let $\Upsilon \in \overline{k}[\vvar_1,\ldots,\vvar_n,\wvar_2,\ldots,\wvar_n]$ be the polynomial obtained from Lemma \ref{lem:upsilon}. If $(\bv,\bw) \in \F_q^{2n-1}$ satisfy $\Upsilon(\v_1,\ldots,\v_n,\w_2,\ldots,\w_n) \neq 0$, then
	\[f(1,X+\v_1,\w_2X +\v_2,\ldots,\w_nX +\v_n) = \prod_{1 \leq i \leq t} f_i(1,X+\v_1,\w_2X +\v_2,\ldots,\w_nX +\v_n)\]
	is squarefree of degree $d$ in $k[X]$.
	
	For $0 \leq D_i < d_i$, and any root $\alpha_{i,0}$ of $f_i(1,X+\v_1,\w_2X +\v_2,\ldots,\w_nX +\v_n)$, we apply Theorem \ref{thm:effBertini_modified} to find $\Psi_{i,D_i}^{(\alpha_{i,0})} \in \overline{k}[\zvar_2,\ldots,\zvar_n]$ so that whenever $\Psi_{i,D_i}^{(\alpha_{i,0})}(\z_2,\ldots,\z_n) \neq 0$, we have $(f_i)_{\bv,\bw,\bz}$ has no nonconstant factor of degree at most $D$ over $\overline{k}$ vanishing on $(\alpha_{i,0},0)$. 
	
	We now address parts (i) and (ii) separately.
	\begin{enumerate}[label = (\roman*)]
		\item Set $\Psi = \prod_{1 \leq i \leq t} \Psi_{i,d_i-1}^{(\alpha_{i,0})}$ which satisfies
		\[\deg \Psi = \sum_{1 \leq i \leq t} \frac{d_i}{8}\left(3d_i^3 - 2d_i^2 - 3d_i + 2\right).\] In this case, if $\Psi(\z_2,\ldots,\z_n) \neq 0$, we have $(f_i)_{\bv,\bw,\bz}$ is irreducible of degree $d_i$ over $\overline{k}$ for all $i$, hence $\fvwz$ has type $(d_1,\ldots,d_t)$. 
		
		\item Set 
		\[\Psi = \prod_{1 \leq i \leq t} \prod_{\substack{\alpha_{i,0} \in \overline{k} \\ (f_i)_{\bv,\bw,\bz}(\alpha_{i,0},0) = 0}} \Psi_{i,D_i}^{(\alpha_{i,0})},\]
        where in the innermost product, $\alpha_{i,0}$ runs over the $d_i$ roots of $f_i(1,X + \v_1, \w_2X + \v_2,\ldots,\w_n X + \v_n)$. This satisfies
		\[\deg \Psi \leq \sum_{1 \leq i \leq t} \frac{d_iD_i}{8}\left(-D_i^3 + 4d_iD_i^2 - 6D_i^2 + 12d_iD_i - 11D_i + 8d_i - 6\right).\]
		In this case, if $\Psi(\z_2,\ldots,\z_n) \neq 0$, we have each $(f_i)_{\bv,\bw,\bz}$ has no nonconstant factor of degree at most $D_i$ through any of the $d_i$ roots $\alpha_{i,0}$. Thus $(f_i)_{\bv,\bw,\bz}$ has no nonconstant factor of degree at most $D_i$.
	\end{enumerate}
Note that in either case, while $\Psi$ depends on the choice of $(\bv,\bw)$, the upper bound for its degree does not.
	
In either case, we claim that for at most $(\deg \Upsilon + \deg \Psi)q^{3n-3}$ tuples $(\bv,\bw,\bz)$ does $\fvwz$ fail to satisfy the desired property. Set
	\[N_\Upsilon = \# \left\{ (\bv,\bw) \in \F_q^{2n-1} : \Upsilon(\v_1,\ldots,\v_n,\w_2,\ldots,\w_n) = 0 \right\}\]
	which is bounded above by $(\deg \Upsilon)q^{2n-2}$ (see e.g.\ \cite[Lemma 2.1]{CafureMatera}).
	Then the number of $(\bv,\bw,\bz) \in \F_q^{3n-2}$ for which $\fvwz$ fails to have the desired property is bounded above by
	\begin{align*} 
		N_\Upsilon q^{n-1} + (q^{2n-1} - N_\Upsilon) \deg \Psi q^{n-2}
		\leq \left(\deg \Upsilon + \deg \Psi\right)q^{3n-3}.
	\end{align*}
	Substituting $\deg \Upsilon \leq 2d^2$ and the appropriate degree bound for $\Psi$ yields (i) and (ii).
\end{proof}

\begin{proof}[Proof of Theorem \ref{thm:effBertini_existence}]
	When \qbound{2d^2}, there exists $(\bv,\bw) \in \F_q^{2n-1}$ such that $\Upsilon(\v_1,\ldots,\v_n,\w_2,\ldots,\w_n) \neq 0$. As in the proof of Theorem \ref{thm:effBertini}, we define $\Psi \in \overline{k}[\zvar_2,\ldots,\zvar_n]$ appropriately for (i) and (ii) using Theorem \ref{thm:effBertini_modified}. 
	
	If $q > \deg \Psi$, there exists $(\bv,\bw,\bz)$ such that $\fvwz$ has the same type as $f$ (resp.\ $(f_i)_{\bv,\bw,\bz}$ has no nonconstant factor of degree at most $D_i$). Since we only need the one $\fvwz$, the degree of $\Upsilon$ does not contribute to this estimate, and the degree estimates for $\Psi$ alone yield (i) and (ii).
\end{proof}

\section{Degree \texorpdfstring{$d$}{d} forms over \texorpdfstring{$\F_q$}{Fq}}
\label{sec:facts_degreed_Fq}

In this section, we record some generalities about degree $d$ forms $f(x_0,\ldots,x_n)$  and plane curves defined over $\F_q$ before specializing to degrees $d=7$ and $d=11$ in \S \ref{subsec:d=7} and \S \ref{subsec:d=11}, respectively.

\begin{lemma}
\label{lem:all_pts_vanish}
	Suppose \qboundgeq{d} and $f$ is a nonzero degree $d$ form in $n+1 \geq 2$ variables. Then $X_f(\F_q) \subsetneq \P^n(\F_q)$.
\end{lemma}

\begin{proof}
	Suppose $n=1$, i.e.\ $f$ is a binary form of degree $d$ that vanishes on all $q+1$ points of $\P^1(\F_q)$. Then $x_0^qx_1 - x_0x_1^q \mid f$, so we have $\deg f \geq q+1 > d$, a contradiction.
	
	We proceed by induction on $n$. If $f$ vanishes on all points in $\P^n(\F_q)$, then $f(x_0, \ldots, x_{n-1},0)$ is a (possibly now zero) degree $d$ form in $n$ variables $x_0,\ldots,x_{n-1}$ that vanishes on all $\F_q$-points of $\P^{n-1}$. By the induction hypothesis, it must indeed be the zero form, i.e.\ $x_n \mid f$.
	
	Repeating for all possible distinct (up to scaling) linear factors of $f$, of which there are more than $d$, we have found too many linear factors of $f$, forcing $f=0$, a contradiction.
\end{proof}

A consequence of Lemma \ref{lem:all_pts_vanish} is that when \qboundgeq{d} and $H \subset \P^n$ is a hyperplane defined over $\F_q$, $X_f$ vanishes on \textit{all $\F_q$-points of} $H$ if and only if $H \subseteq X_f$, i.e.\ $f$ has a linear factor which corresponds to $H$. Moving forward, we will often conflate these notions. 

Let us set a bit more notation. We refer to subvarieties $U \subset \P^n$ cut out by the vanishing of a collection of linear forms as \textit{linear subvarieties}. Given such $U$ defined over $\F_q$ of dimension $s$, we use $f|_U$ to refer to the form cutting out $X_f \cap U$. If $U \subset X_f$, then $f|_U = 0$; otherwise $f|_U$ is a nonzero degree $d$ form in $s+1$ variables. Note that $f|_U$ is only well defined up to choice of coordinates on $U$ and scaling, but this will not be of concern. 

Given linear subvarieties $U,V \subseteq \P^n$, their \textit{span}, denoted $\Span{U,V}$, is the minimal linear subvariety of $\P^n$ containing $U$ and $V$. In particular, when $U$ and $V$ are $\F_q$-points, $\Span{U,V}$ denotes the line between them.

Before specializing to $d=7,11$, we record a flexible version of the Hasse--Weil bound for counting smooth $\F_q$-points on a plane curve $C$ which allows for $C$ to have singularities or be reducible.

\begin{lemma}
\label{lem:finding_sm_pt}
    Let $C \subset \P^2$ be a plane curve defined over $\F_q$ of degree $d$. Suppose $C' \subset C$ is a geometrically integral component, defined over $\F_q$, of multiplicity 1 and degree $d'$. Denote by $g'$ the geometric genus of the normalization of $C'$. Then we have
    \begin{equation}\label{eq:ns_pt_bound}
        \#C(\F_q)^{\sm} \geq q+1 - g' \lfloor 2\sqrt{q} \rfloor - (d'-1)(d'-2) + 2g' - (d-d')d'.
    \end{equation}
\end{lemma}

\begin{proof}
    We begin by counting points on $C'$. The usual Hasse--Weil bounds can be modified to allow for $C'$ to be singular, yielding
    \begin{equation}\label{eq:pt_ct_allowsing}
        \#C'(\F_q) \geq q+1 - g'\lfloor 2\sqrt{q} \rfloor - \frac12(d'-1)(d'-2) + g'
    \end{equation}
    (see e.g.\ \cite[Corollary 1]{LeepYeomans1994_singular_curves}). There are at most $\frac12(d'-1)(d'-2) - g'$ singular $\F_q$-points on $C'$, so from \eqref{eq:pt_ct_allowsing} we deduce
    \begin{align}\label{eq:pt_ct_sm}
        \nonumber \#C'(\F_q)^\sm &= \#C'(\F_q) - \#C'(\F_q)^{\mathrm{sing}}\\
        &\geq q+1 - g'\lfloor 2\sqrt{q} \rfloor - (d'-1)(d'-2) + 2g'.
    \end{align}

    A point in $C'(\F_q)^{\sm}$ is a smooth point of $C$ if it does not also lie on $C - C'$. By Bezout's theorem, there are at most $d'(d-d')$ such intersection points. Thus \eqref{eq:ns_pt_bound} follows from subtracting this from \eqref{eq:pt_ct_sm}.    
\end{proof}

\begin{remark}
    For constant $d'$ and \qbound{2}, the lower bound \eqref{eq:ns_pt_bound} is monotonically decreasing in $g'$. In other words, for $C'$ of fixed degree, the $q$-threshhold where the lower bound becomes positive is largest when $g'$ is as large as possible. 
\end{remark}

\subsection{\texorpdfstring{$d=7$}{d=7}}
\label{subsec:d=7}

We now specialize to the case of $d=7$, where $f(x_0,\ldots,x_n)$ is a degree 7 form over $\F_q$. 

\begin{lemma}
\label{lem:too_many_linears_d=7}
    Let \qboundgeq{7}. Suppose $f(x_0,\ldots,x_n)$ is a degree 7 form and there exist $k$ distinct linear subvarieties $U_1,\ldots,U_k \subset X_f \subset \P^n$ of dimension $s$ such that $\dim \Span{U_1,\ldots,U_k} = s+1$.
    \begin{enumerate}[label = (\roman*)]
        \item If $k \geq 4$ then $X_f(\F_q)^\sm \neq \emptyset$ or $\Span{U_1,\ldots,U_k} \subset X_f$.
        \item If $k =3$ and there exists $v \in X_f \cap \Span{U_1,U_2,U_3}$ such that $v \notin U_1,U_2,U_3$, then $X_f(\F_q)^\sm \neq \emptyset$ or $\Span{U_1,U_2,U_3} \subset X_f$.
    \end{enumerate}
\end{lemma}

\begin{proof}
    If $\Span{U_1,\ldots,U_k} \not\subset X_f$, then $f|_{\Span{U_1,\ldots,U_k}}$ is nonzero. In this case, each $U_i$ corresponds to a distinct linear factor of $f|_{\Span{U_1,\ldots,U_k}}$ defined over $\F_q$. 
    
    Both statements follow from the fact that when \qboundgeq{7}, a degree 7 form with a linear factor of multiplicity one has a smooth solution. If $U_1 \subset X_f \cap \Span{U_1,\ldots,U_k}$ corresponds to such a factor of $f|_{\Span{U_1,\ldots,U_k}}$, then 
    \[\#U_1(\F_q) = \frac{q^{s+1} - 1}{q-1} \quad\text{and}\quad \#(U_1 \cap ((X_f \cap \Span{U_1,\ldots,U_k}) - U_1))(\F_q) \leq 6\frac{q^{s} - 1}{q-1},\]
    so there exists $u \in (U_1 - (X_f - U_1))(\F_q) \subset X_f(\F_q)^\sm$. 

    If $X_f(\F_q) = \emptyset$, then $f|_{\Span{U_1,\ldots,U_k}}$ has no linear factors of multiplicity one, hence it can have at most three linear factors, proving (i). If it has exactly three such factors, then it must have type $(1^3,1^2,1^2)$. This implies $X_f = U_1 \cup U_2 \cup U_3$, proving (ii) by contrapositive.
\end{proof}

We next give an intermediate result in a similar spirit to \cite[Lemma 5.3]{LeepYeomans1996_quintic_forms}.

\begin{lemma}
\label{lem:7.3}
    Let \qboundgeq{7}. Suppose $U \subset X_f$ is a linear subvariety of positive dimension defined over $\F_q$, and $v \in X_f(\F_q) - U$. If there exist distinct codimension 1 linear subvarieties $U_1^{(1)}, U_2^{(1)}, U_3^{(1)} \subset U$ such that
    \[\Span{U_i^{(1)}, v} \subset X_f \text{ for } i =1,2,3,\]
    then either $X_f(\F_q)^{\sm} \neq \emptyset$ or $\Span{U,v} \subset X_f$.
\end{lemma}

\begin{proof}
    Let $s = \dim U \geq 1$. If $f|_{\Span{U,v}} \neq 0$, we apply Lemma \ref{lem:too_many_linears_d=7}(i) to $U, \Span{U_1^{(1)},v}, \Span{U_2^{(1)},v}, \Span{U_3^{(1)},v}$ within $\Span{U,v} \simeq \P^{s+1}$. If $f|_{\Span{U,v}} = 0$ then $\Span{U,v} \subset X_f$.
\end{proof}

Next we examine the situation of a plane curve of degree 7 over $\F_q$ which has at least three non-colinear $\F_q$-points and no linear components defined over $\F_q$. Unfortunately, we cannot guarantee the existence of a smooth $\F_q$-point, even when $q$ is sufficiently large. We can, however, find either a smooth $\F_q$-point or a line spanned by (singular) $\F_q$-points which elsewhere meets the curve at a nonsingular geometric point.

\begin{lemma}
\label{lem:plane_with_3pts}
    Let \qbound{591}. Suppose $f \in \F_q[x_0, x_1, x_2]$ is a nonzero ternary degree 7 form with no linear factors defined over $\F_q$. Suppose further that there exist $u,v,w \in X_f(\F_q)$ such that $\Span{u,v,w} = \P^2$. Then either $X_f(\F_q)^{\sm} \neq \emptyset$ or there exist $y,z \in X_f(\F_q)$ such that $\Span{y,z} \not\subset X_f$ and $X_f \cap \Span{y,z}$ contains a geometric point of multiplicity 1.
\end{lemma}

\begin{proof}
    If any of $u,v,w \in X_f(\F_q)^\sm$, then we are done, so we assume they are all singular. Since $f$ has no linear factors defined over $\F_q$, the possible factorization types of $f$ over $\F_q$ are $(7), (5, 2), (4, 3), (3, 2, 2), (3, 2^2)$. 

    Suppose one of these factors of multiplicity 1 remains irreducible over $\overline{\F_q}$. Let $C' \subset X_f$ denote the corresponding irreducible component curve of degree $d'$, defined over $\F_q$ and absolutely irreducible. If $d' = 7$ then the geometric genus of the normalization of $X_f$ is at most 12, owing to the presence of the singular points $u,v,w$. Lemma \ref{lem:finding_sm_pt} gives that \qbound{591} guarantees $X_f(\F_q)^{\sm} \neq \emptyset$. If $d' < 7$ then the hypothesis \qbound{591} is sufficient to show $X_f(\F_q)^\sm \neq \emptyset$ by Lemma \ref{lem:finding_sm_pt}.

    Assume now that all of the $\F_q$-factors of $f$ with multiplicity 1 are reducible over $\overline{\F_q}$. In the case of factorization type (7), the only possibility is that $f$ factors into 7 linear forms conjugate over $\F_{q^7}$. Such $f$ can vanish on at most one $\F_q$-point, contradicting our hypothesis that $u,v,w \in X_f(\F_q)$. Similarly, in the case of $(5,2)$ we have $f$ can vanish on at most two $\F_q$-points, again contradicting the hypothesis.

    In the remaining factorization types we see that $f$ has a cubic factor which is geometrically reducible, necessarily factoring as a product of three linear forms conjugate over $\F_{q^3}$. Let $L_1, L_2, L_3$ be the corresponding lines in the plane defined over $\F_{q^3}$.

    Suppose $L_1, L_2, L_3$ meet at a common point $x \in X_f(\F_q)$. Since the points $u,v,w$ are not collinear, we must have that $x$ does not lie on at least one of the lines $\Span{u,v}, \Span{u,w}, \Span{v,w}$. Without loss of generality, assume $x \notin \Span{u,v}$. Thus $f|_{\Span{u,v}}$ vanishes to order 2 on $u$ and $v$ and order 1 on the three distinct $\F_{q^3}$-points $\Span{u,v} \cap L_i$, so we may take $y = u$ and $z = v$.

    If $L_1, L_2, L_3$ do not meet at a common point, then they contain no $\F_q$-points. Again considering the line $\Span{u,v}$, we have that at most two of of the $\F_{q^3}$-points $\Span{u,v} \cap L_i$ can coincide. Thus $f|_{\Span{u,v}}$ vanishes to order 2 on $u$ and $v$ and order 1 on at least one of the $\Span{u,v} \cap L_i$, so we may take $y = u$ and $z = v$.
\end{proof}

We conclude by giving what amounts to a higher dimensional version of Lemma \ref{lem:plane_with_3pts}.

\begin{proposition}
\label{prop:codim3bert_d=7}
    Let \qbound{591}. Suppose $f \in \F_q[x_0,\ldots,x_n]$ is a degree 7 form defined over $\F_q$. Let $U, V, W \subset X_f$ be distinct linear subvarieties of dimension $s$ such that $\dim U \cap V \cap W = s-1$, $\dim \Span{U,V,W} = s+2$, and $f|_{\Span{U,V,W}}$ is nonzero with no linear factors defined over $\F_q$. Then $X_f(\F_q)^\sm = \emptyset$ or there exist $y,z \in X_f(\F_q)$ such that $\Span{y,z} \not\subset X_f$ and $X_f \cap \Span{y,z}$ contains a geometric point of multiplicity 1.
\end{proposition}

\begin{proof}
    The $s=0$ case is Lemma \ref{lem:plane_with_3pts}. Assume $s \geq 1$. Replacing $f$ by $f|_{\Span{U,V,W}}$, we may assume $n=s+2 \geq 3$, $\Span{U,V,W} = \P^n$, and $f$ has no linear factors defined over $\F_q$. 
    
    After a change of coordinates, we may assume the intersection $U \cap V \cap W$, codimension 3 in $\P^n$, is given by
    \[U \cap V \cap W \colon x_{n-2} = x_{n-1} = x_n = 0.\]
    For a generic plane $P \subset \P^n$, we have $P \cap U \cap V \cap W = \emptyset$. Moreover, this holds for any plane $P$ given by
    \[P \colon \begin{cases}
        x_0 = Z, \\
        x_1 = X +\v_1Z, \\
        x_2 =\w_2X + \z_2Y +\v_2Z, \\
        \vdots \\
        x_n =\w_nX + \z_nY +\v_nZ,
    \end{cases} \text{ satisfying } \det \begin{pmatrix}
		\v_{n-2} &\w_{n-2} &\z_{n-2} \\
		\v_{n-1} &\w_{n-1} &\z_{n-1} \\
		\v_{n} &\w_{n} &\z_{n} 
	\end{pmatrix} \neq 0.\]
    When $n=3$, we set $w_{n-2}=1$ and $z_{n-2}=0$ above. 

    Let $\Upsilon' \in \F_q[\vvar_1,\ldots,\vvar_n,\wvar_2,\ldots,\wvar_n]$ be the generic determinant above,
    \begin{equation}
    \label{eq:precond_v,w}
    \Upsilon' = (\vvar_{n-1}\wvar_n - \vvar_n\wvar_{n-1})(\vvar_{n-2}\wvar_n - \vvar_n\wvar_{n-2})(\vvar_{n-2}\wvar_{n-1} - \vvar_{n-1}\wvar_{n-2}).
    \end{equation}
    For any $(\bv,\bw)$ satisfying $\Upsilon'(\v_1,\ldots,\v_n,\w_2,\ldots,\w_n) \neq 0$, we define $\Xi \in \F_q[\zvar_2,\ldots,\zvar_n]$ by
    \[\Xi = \zvar_{n-2}(\v_{n-1}\w_n -\v_n\w_{n-1}) - \zvar_{n-1}(\v_{n-2}\w_n -\v_n\w_{n-2}) + \zvar_n(\v_{n-2}\w_{n-1} -\v_{n-1}\w_{n-2}).\]
    which is necessarily nonzero.

    We now apply an effective Bertini argument which is essentially a minor modification of Theorem \ref{thm:effBertini_existence}(ii). Let $\Upsilon$ be the polynomial of degree 98 constructed in Lemma \ref{lem:upsilon}. Since \qbound{104}, there exist $(\bv,\bw)$ such that $\Upsilon(\bv,\bw)\Upsilon'(\bv,\bw) \neq 0$. 

    Suppose $f$ has factorization type $(d_1^{e_1},\ldots,d_t^{e_t})$ over $\overline{\F_q}$. For $1 \leq i \leq t$, set
    \[D_i = \begin{cases} 0 & d_i = 1, \\ 1 & d_i > 1. \end{cases}\]
    Consider 
    \[\Psi = \prod_{1 \leq i \leq t} \prod_{1 \leq j \leq d_i}\Psi^{(\alpha_{j,0})}_{i,D_i}\]
    as in the proof of Theorem \ref{thm:effBertini}(ii). We have
    \[\deg \Psi \Xi \leq 1 + \sum_{1 \leq i \leq t} \frac{d_iD_i}{8}\left(-D_i^3 + 4d_iD_i^2 - 6D_i^2 + 12d_iD_i - 11D_i + 8d_i - 6\right).\]
    In particular, when \qbound{\deg \Psi\Xi}, there exists $\bz \in \F_q^{n-1}$ such that $\Psi(\bz)\Xi(\bz) \neq 0$. A straightforward computation shows $\deg \Psi \Xi \leq 127$, so our hypotheses on $q$ more than suffice for this purpose.

    Since $\Xi(\bv,\bw,\bz) \neq 0$, the plane $P$ satisfies $P \cap U \cap V \cap W = \emptyset$. In particular, the points
    \[u = P \cap U, v = P \cap V, w = P \cap W\]
    are distinct, noncolinear points in $(X_f \cap P)(\F_q)$. Moreover, since $f|_P$ has the same factorization type as $\fvwz = 0$, and the latter only has linear factors inherited from $f$ itself which remain distinct in $\fvwz$, any linear factors of $f|_P$ are defined over $\overline{\F_q}$ but not $\F_q$. Thus $f|_P$ satisfies the hypotheses of Lemma \ref{lem:plane_with_3pts} and we are done.
\end{proof}

\subsection{\texorpdfstring{$d=11$}{d=11}}
\label{subsec:d=11}

The same arguments in degree 11 give analogues of Lemmas \ref{lem:too_many_linears_d=7}, \ref{lem:7.3}, which we use to give analogues of Lemma \ref{lem:plane_with_3pts} and Proposition \ref{prop:codim3bert_d=7}.

\begin{lemma}
\label{lem:too_many_linears_d=11}
    Let \qboundgeq{11}. Suppose $f(x_0,\ldots,x_m)$ is a degree 11 form and there exist $k$ distinct linear subvarieties $U_1,\ldots,U_k \subset X_f \subset \P^n$ of dimension $s$ such that $\dim \Span{U_1,\ldots,U_k} = s+1$.
    \begin{enumerate}[label = (\roman*)]
        \item If $k \geq 6$ then $X_f(\F_q)^\sm \neq \emptyset$ or $\Span{U_1,\ldots,U_k} \subset X_f$.
        \item If $k =5$ and there exists $v \in X_f \cap \Span{U_1,\ldots,U_5}$ such that $v \notin U_1,\ldots,U_5$, then $X_f(\F_q)^\sm \neq \emptyset$ or $\Span{U_1,\ldots,U_5} \subset X_f$.
    \end{enumerate}
\end{lemma}

\begin{lemma}
\label{lem:11.3}
    Let \qboundgeq{11}. Suppose $U \subset X_f$ is a linear subvariety of positive dimension defined over $\F_q$, and $v \in X_f(\F_q) - U$. If there exist distinct codimension 1 linear subvarieties $U_1^{(1)}, \ldots, U_5^{(1)} \subset U$ such that
    \[\Span{U_i^{(1)}, v} \subset X_f \text{ for } i =1,\ldots,5,\]
    then either $X_f(\F_q)^{\sm} \neq \emptyset$ or $\Span{U,v} \subset X_f$.
\end{lemma}

\begin{lemma}
\label{lem:plane_with_3pts_d=11}
    Let \qbound{7061}. Suppose $f \in \F_q[x_0, x_1, x_2]$ is a nonzero ternary degree 11 form with no linear factors defined over $\F_q$. Suppose further that there exist $u,v,w \in X_f(\F_q)$ such that $\Span{u,v,w} = \P^2$. Then either $X_f(\F_q)^{\sm} \neq \emptyset$ or there exist $y,z \in X_f(\F_q)$ such that $\Span{y,z} \not\subset X_f$ and $X_f \cap \Span{y,z}$ contains a geometric point of multiplicity 1.
\end{lemma}

\begin{proof}
    The possible factorization types of $f$ over $\F_q$ are
    \begin{align*}
        (11), (9,2), (8,3), 
        (7,4), (7,2,2), (7,2^2), 
        (6,5), (6,3,2), \\
        (5,4,2), (5,3,3), (5,3^2), (5,2,2,2), (5,2^2,2), (5,2^3),
        (4,4,3), (4^2,3), (4,3,2,2), (4,3,2^2),\\
        (3,3,3,2), (3^2,3,2), (3^3,2), (3,2,2,2,2), (3,2^2,2,2), (3,2^2,2^2), (3,2^3,2), \text{ or } (3,2^4).
    \end{align*}
    If any of the factors of multiplicity 1 defined over $\F_q$ are irreducible over $\overline{\F_q}$, then we can use Lemma \ref{lem:finding_sm_pt} to find $X_f(\F_q)^\sm \neq \emptyset$, as in the proof of Lemma \ref{lem:plane_with_3pts}. The limiting case is that of $f$ geometrically irreducible, where the presence of the three singular $\F_q$-points $u,v,w$ means \qbound{7061} suffices.

    Write the factorization over $\F_q$ as $f = \prod_{1 \leq i \leq t} g_i^{e_i}$, i.e.\ if $f$ has type $(d_1^{e_1},\ldots,d_t^{e_t})$ over $\F_q$, then $g_i$ is defined over $\F_q$ and has degree $d_i$. Assume for all $i$ with $e_i = 1$, we have $g_i$ is geometrically reducible; otherwise we are in the situation above.
    
    Suppose there exists an $i$ with the following properties:
    \begin{enumerate}[label = (\roman*)]
        \item $d_i$ is prime,
        \item $e_i = 1$,
        \item $d_i > \frac{d_j}{2}$ for all $j \neq i$ with $e_j = 1$,
        \item $d_i > d_j$ for all $j \neq i$ with $e_j > 1$.
    \end{enumerate}
    Since $g_i$ is reducible over $\overline{\F_q}$, it must factor into $d_i$ linear forms defined over $\F_{q^{d_i}}$, cyclically permuted by the action of $\mathrm{Gal}(\F_{q^{d_i}}/\F_q)$. In particular, $g_i$ vanishes on at most one $\F_q$-point, so without loss of generality it vanishes on $d_i$ distinct points on the line $\Span{u,v}$ defined (and conjugate) over $\F_{q^{d_i}}$.

    Consider now how the $g_j$ vanish on $\Span{u,v}$ for $j \neq i$. If $e_j = 1$, then $g_j$ --- and its restriction to the line $\Span{u,v}$ --- factors into forms of degree at most $d_j/2$ over $\overline{\F_q}$. By (iii), the degree of these factors of $g_j$ is too small to vanish on any $\F_{q^{d_i}}$-points of the line. The same argument works for $e_j > 1$, since we have the stronger hypothesis (iv). Thus $u,v \in X_f(\F_q)$ satisfy the desired requirements, with the zeros of $g_i$ on this line providing the necessary geometric point(s) of multiplicity 1. 

    Conditions (i) -- (iv) are satisfied for all types except $(11)$, $(9,2)$, $(8,3)$, $(6,3,2)$, $(4^2,3)$, $(3^2,3,2)$, and $(3^3,2)$. We deal with the stragglers ad hoc.

    \begin{itemize}\setlength{\itemsep}{1ex}
        \item[$(11)$] If $f$ is geometrically reducible, then it must factor into 11 conjugate linear forms defined over $\F_{q^{11}}$. However, this is not possible, as such $f$ vanishes on at most one $\F_q$-point.

        \item[$(9,2)$] If $g_1$ is not geometrically irreducible, it must factor into three conjugate cubic forms $h,h',h''$ over $\F_{q^3}$. As above, we may assume the quadratic factor $g_2$ vanishes on two distinct $\F_{q^2}$-points of $\Span{u,v}$. Since $g_1$ must then vanish on both $u,v \in X_f(\F_q)$, each of $h,h',h''$ must vanish on both $u$ and $v$, and hence at most one more $\F_q$-point on $\Span{u,v}$, so the line $\Span{u,v}$ satisfies the desired requirements.

        \item[$(8,3)$] If $g_1$ is not geometrically irreducible, it must factor into two conjugate quartic forms $h,h'$ over $\F_{q^2}$, while we may assume the cubic factor $g_2$ vanishes on three distinct conjugate $\F_{q^3}$-points of $\Span{u,v}$. As in the $(9,2)$ case, $h,h'$ are forced to vanish on $u,v$, so they cannot vanish on any $\F_{q^3}$-points along $\Span{u,v}$.

        \item[$(6,3,2)$] Here $g_1$ can factor in two ways. Suppose it factors into conjugate cubic forms $h,h'$ defined over $\F_{q^2}$ and assume the cubic factor $g_2$ vanishes on three distinct $\F_{q^3}$-points of $\Span{u,v}$. The quadratic factor $g_3$ vanishes on at most one of $u$ or $v$, thus $h,h'$ must both vanish on (at least) the other. Their degrees are then too small to also vanish on any $\F_{q^6}$-points on $\Span{u,v}$, nor in particular any $\F_{q^3}$-points, so we are done.

        If instead $g_1$ factors into conjugate quadratic forms $h,h',h''$ defined over $\F_{q^3}$, a similar argument shows they all must vanish on at least one of $u$ or $v$ and one other $\F_q$-point on the line $\Span{u,v}$, while the quadratic factor $g_3$ vanishes on distinct $\F_{q^2}$-points of $\Span{u,v}$. The cubic factor $g_2$ meets the line at an $\F_q$-point or distinct $\F_{q^3}$-points, but eitherway the requirements of the lemma are met.

        \item[$(4^2,3)$] Assume the cubic factor $g_2$ meets the line $\Span{u,v}$ at three distinct $\F_{q^3}$-points. For degree reasons only, the quartic factor $g_1$ can vanish on at most two of them, since it must vanish on $u$ and $v$.
        
        \item[$(3^2,3,2)$] Suppose the cubic factor $g_1$ is irreducible over $\overline{\F_q}$. The cubic factor $g_2$ vanishes on at most one $\F_q$-point in the plane. Let $y,z \in X_{g_1}(\F_q)$ be distinct points such that $\Span{y,z}$ avoids this point (if it exists); note that the hypotheses on $q$ ensure more than enough $\F_q$-solutions to $g_1=0$ for this. On the line $\Span{y,z}$, $g_1$ vanishes on $y$, $z$, and at least one more point, necessarily defined over $\F_q$; the cubic factor $g_2$ vanishes in three distinct $\F_{q^3}$-points, while the quadratic $g_3$ vanishes on an $\F_q$-point or two $\F_{q^2}$-points. In either case, the zeros of $g_2$ along this line provide the necessary points of multiplicity 1. 

        If $g_1$ is reducible over $\overline{\F_q}$, then in fact $X_f(\F_q) = \{u,v,w\}$, with each of $g_1, g_2, g_3$ vanishing on exactly one. If $u \in X_{g_1}$, then either of $\Span{u,v}$ or $\Span{u,w}$ satisfy the desired requirements.

        \item[$(3^3,2)$] Assume the quadratic factor $g_2$ meets the line $\Span{u,v}$ at two distinct $\F_{q^2}$-points. As in the $(4^2,3)$ case, the cubic factor $g_1$ can vanish on at most one of them, since it must vanish on $u$ and $v$.
\end{itemize}\vspace{-1ex}
\end{proof}

\begin{proposition}
\label{prop:codim3bert_d=11}
    Let \qbound{7061}. Suppose $f \in \F_q[x_0,\ldots,x_n]$ is a degree 11 form defined over $\F_q$. Let $U, V, W \subset X_f$ be distinct linear subvarieties of dimension $s$ such that $\dim U \cap V \cap W = s-1$, $\dim \Span{U,V,W} = s+2$, and $f|_{\Span{U,V,W}}$ is nonzero with no linear factors defined over $\F_q$. Then $X_f(\F_q)^\sm = \emptyset$ or there exist $y,z \in X_f(\F_q)$ such that $\Span{y,z} \not\subset X_f$ and $X_f \cap \Span{y,z}$ contains a geometric point of multiplicity 1.
\end{proposition}

\begin{proof}
    The proof is essentially identical to that of Proposition \ref{prop:codim3bert_d=7}, except that we reduce to Lemma \ref{lem:plane_with_3pts_d=11} instead of Lemma \ref{lem:plane_with_3pts}. This accounts for the hypothesis \qbound{7061}. We briefly sketch the other differences.
    
    The polynomial $\Upsilon$ from Lemma \ref{lem:upsilon} now has degree at most 242, while $\Upsilon'$ is the same as in the proof of Proposition \ref{prop:codim3bert_d=7}. Thus there exists $(\bv,\bw) \in \F_q^{2m-2}$ for which $\Upsilon\Upsilon' \neq 0$. 
    
    With $\Psi$ defined the same way as in the proof of Proposition \ref{prop:codim3bert_d=7}, its degree is at most 330 by Theorem \ref{thm:effBertini_modified}. The polynomial $\Xi$ is again linear. Thus $q$ is sufficiently large to guarantee the existence of $\bz$ with $\Psi \Xi \neq 0$. 

    Finally, $\fvwz$ has no linear factors defined over $\F_q$, corresponding to a plane $P$ not containing $U \cap V \cap W$ such that $X_f \cap P$ has no linear components defined over $\F_q$. Since $P$ meets $U, V, W$ in distinct noncolinear points $u,v,w$, the hypotheses of Lemma \ref{lem:plane_with_3pts_d=11} are satisfied, completing the proof.
\end{proof}

\section{Reduced forms and consequences}
\label{sec:reduced_forms_and_consequences}

We recall the setting of Artin's conjecture for degree $d$ forms. $K/\Q_p$ is a $p$-adic field with valuation ring $\O_K$, uniformizer $\pi$, normalized discrete valuation $v$, and residue field $\O_K/(\pi) \simeq \F_q$. We have a homogeneous form $F \in K[x_0, \ldots, x_n]$ of degree $d$ for $n \geq d^2$. After rescaling, we may assume $F \in \O_K[x_0, \ldots, x_n]$ such that $\overline{F} \neq 0$ in $\F_q[x_0, \ldots, x_n]$. We will use the lowercase $f = \overline{F}$ to denote the reduction to the residue field $\F_q$, and $X_f/\F_q$ for the special fiber of $X_F$ over $\O_K$.

By Hensel's lemma, if $X_f(\F_q)^{\sm} \neq \emptyset$, then $X_F(K) \neq \emptyset$. Therefore our goal will be to show that $X_f$ has a smooth $\F_q$-point when $q$ is sufficiently large. A key strategy in achieving this was introduced by Birch and Lewis \cite{BirchLewis1959} in the degree $d=5$ case, then later refined by Laxton and Lewis \cite{LaxtonLewis1965} to apply in degrees $d=7, 11$. They define a notion of \textit{reduced forms} which we recall now before establishing the relevant consequences. Both \cite{LaxtonLewis1965} and \cite[\S 4]{LeepYeomans1996_quintic_forms} are useful references.

\subsection{Reduced forms}
\label{sec:reduced_forms}

With $F$ as above, let 
\[I(F) = \mathrm{Res}\left(\frac{\partial F}{\partial x_0}, \ldots, \frac{\partial F}{\partial x_n}\right) \in \O_K.\]
By the general theory of (multivariate) resultants, $I(F)$ is homogeneous of degree $(d-1)^{n}$ in the coefficients of the partial derivatives of $F$. We also have \cite[Lemma 5]{LaxtonLewis1965}
\[I(aF(T\mathbf{x})) = a^{(n+1)(d-1)^n}(\det T)^{d(d-1)^n}I(F)\]
for $a \in K^\times$ and $T \in \GL_{n+1}(K)$ acting on the variables $(x_0, \ldots, x_n)$ by linear change of coordinates.

\begin{definition}[equivalent forms]
    Two degree $d$ forms $F,F' \in \O_K[x_0, \ldots, x_n]$ are \textbf{equivalent} if there exists $a \in K^\times$ and $T \in \GL_{n+1}(K)$ such that $F' = aF(T\mathbf{x})$.
\end{definition}

\begin{definition}[reduced form]
    We say a degree $d$ form $F \in \O_K[x_0, \ldots, x_n]$ is \textbf{reduced} if $I(F) \neq 0$ and 
    $v(I(F)) \leq v(I(aF(T\mathbf{x})))$ for all $a \in K^\times$ and $T \in \GL_{n+1}(K)$ such that $aF(T\mathbf{x})$ has coefficients in $\O_K$. That is, $I(F)$ has minimal valuation among equivalent forms with coefficients in $\O_K$. 
\end{definition}

\begin{proposition}[{\cite{LaxtonLewis1965, LeepYeomans1996_quintic_forms}}]
\label{prop:reduced_facts}
    We have the following facts about reduced forms.
    \begin{enumerate}[label = (\roman*)]
        \item If for all \textit{reduced} degree $d$ forms $F \in \O_K[x_0, \ldots, x_n]$ we have $X_F(K) \neq \emptyset$ then $X_F(K) \neq \emptyset$ for all degree $d$ forms $F \in \O_K[x_0, \ldots, x_n]$, i.e.\ $K$ is $C_2(d)$.

        \item If $n \geq d^2$ and $F$ is a reduced degree $d$ form, then its reduction $f = \overline{F}$ has no linear factors defined over $\overline{\F_q}$.

        \item Let \qboundgeq{d}. If $n \geq d^2$, $F$ is a reduced degree $d$ form, and $f$ vanishes on an $s$-dimensional linear subvariety $U \simeq \P^s$ defined over $\F_q$, then
        \[\#X_f(\F_q) \geq \frac{q^{s+2} - 1}{q-1}.\]
    \end{enumerate}
\end{proposition}

\begin{proof}
    It is established in \cite[Corollary]{LaxtonLewis1965} that to show every $F$ has a $K$-solution, it suffices to show this only for forms with $I(F) \neq 0$. After possibly replacing $F$ by an equivalent form, we may assume $I(F)$ has minimal valuation, i.e.\ $F$ is reduced. This does not change solubility, so it suffices to check $X_F(K) \neq 0$ for reduced forms, giving (i).

    (ii) is \cite[Lemma 9]{LaxtonLewis1965}, while (iii) is \cite[Corollary 4.4]{LeepYeomans1996_quintic_forms}. In both cases, the idea is to use the fact that $F$ is a reduced form to bound the \textit{order} of $f$ (or its factors), i.e.\ to bound the minimal number of variables appearing nontrivially in $f$, up to change of coordinates. To prove (ii), one understands that linear factors have small order which is at odds with these lower bounds; for (iii) an order estimate is combined with Warning's estimate for $\#X_f(\F_q)$ (see e.g.\ \cite[Lemma 3.1]{LeepYeomans1996_quintic_forms}).
\end{proof}

\subsection{Consequences}
\label{subsec:consequences}

Drawing inspiration from the work of Leep and Yeomans in the quintic case \cite{LeepYeomans1996_quintic_forms}, we seek to exploit what we know about reduced forms to produce a nonsingular point of $X_f$ which can be lifted to $X_F(K)$.

Our immediate goal, accomplished below in Proposition \ref{prop:leep_and_yeomanry}, is show that either $X_f(\F_q)^\sm \neq \emptyset$ or we can produce two points $x,y \in X_f(\F_q)$ such that $f$ possesses a nonsingular \textit{geometric} point on the line between them. Later, in the proof of Theorem \ref{thm:deg711}, we will use this in conjunction with Theorem \ref{thm:effBertini_modified} to produce a planar slice of $X_f$ with at least two (possibly singular) $\F_q$-points. The presence of these two $\F_q$-points is key to improving on \cite[Theorem 1]{Wooley2008_7ic_11ic}.

As in \cite{LeepYeomans1996_quintic_forms}, we will make use of a linear subvariety $U \subset X_f$ of maximal dimension. In the following lemma, we show that not only can we find points $v,w \in (X_f - U)(\F_q)$, but we can ensure the line $\Span{v,w}$ is skew to $U$. This is accomplished by strengthening an argument presented in \cite[Proof of Prop.\ 5.4]{LeepYeomans1996_quintic_forms}.

\begin{lemma}
\label{lem:skew_points}
    Let \qbound{d} and suppose $F \in \O_K[x_0, \ldots, x_n]$ is a reduced degree $d$ form for $n \geq d^2$. Suppose $U \subset X_f$ is a linear subvariety defined over $\F_q$ of maximal dimension. Then either $X_f(F_q)^{\sm} \neq \emptyset$ or there exist $v,w \in X_f(\F_q)$ such that 
    \[\Span{v,w} \not\subset X_f \quad\text{and}\quad \Span{v,w} \cap U = \emptyset.\]
\end{lemma}

\begin{proof}
    Let $s = \dim U$. We have $\#U(\F_q) = \frac{q^{s+1}-1}{q-1} < X_f(\F_q)$ by Proposition \ref{prop:reduced_facts}(iii), so there exists $v \in X_f(\F_q) - U(\F_q)$. Moreover, we have $\Span{U,v} \not \subset X_f$ by maximality of $U$, and since \qbound{d} an application of Lemma \ref{lem:all_pts_vanish} reveals that $f$ does not vanish on all $\F_q$-points of $\Span{U,v}$. Since $\#\Span{U,v}(\F_q) = \frac{q^{s+2} - 1}{q-1}$, Proposition \ref{prop:reduced_facts}(iii) again ensures there exists a point $w \in (X_f - \Span{U,v})(\F_q)$.

    We now have $\Span{v,w} \cap U = \emptyset$. If $\Span{v,w} \not\subset X_f$ then we are done. Suppose instead that for all $v,w \in (X_f - U)(\F_q)$ such that $\Span{v,w} \cap U = \emptyset$, we have $\Span{v,w} \subset X_f$. Let $W \subset X_f$ be a linear subvariety over $\F_q$ of maximal dimension such that $W \cap U \subset W$ has codimension at least 2. Such $W$ exists because earlier we found a line $\Span{v,w} \subset X_f$ skew to $U$.\footnote{By convention, for $W$ a linear variety of dimension $s$, the codimension of $\emptyset \subset W$ is $s+1$.}

    We claim that $W$ contains $(X_f - U)(\F_q)$. Suppose to the contrary that $x \in (X_f - U - W)(\F_q)$. Since $U \cap W \subset W$ has codimension at least 2, we have that $\Span{U,x} \not\supset \Span{W,x}$. 

    Let $y \in \Span{W,x}(\F_q)$. If $y \notin \Span{U,x} \cap \Span{W,x}$, then $\Span{x,y} \cap U = \emptyset$. Since $y$ is contained on a line between $x$ and a point of $W$, by assumption we have $\Span{x,y} \subset X_f$, and in particular, $y \in X_f(\F_q)$. Thus $f$ vanishes on all of $\Span{W,x}(\F_q)$, save possibly for a a subspace of codimension at least 1 arising from $\Span{U,x} \cap \Span{W,x}$. But then on $\Span{W,x}$, we have that there exists a nonzero linear form $\ell$ such that $\ell f|_{\Span{W,x}}$ vanishes on all points; by Lemma \ref{lem:all_pts_vanish}, $\ell f|_{\Span{W,x}} = 0$. Since $\ell \neq 0$, we have $f|_{\Span{W,x}} = 0$, i.e.\ $\Span{W,x} \subset X_f$. 
    
    Note that the dimension of $U \cap \Span{W,x}$ is at most one more than that of $U \cap W$, so its codimension is still at most 2 in $\Span{W,x}$. However, this violates the maximality of $W$, yielding a contradiction. We conclude that $(X_f - U)(\F_q) \subset W(\F_q)$, as claimed.

    Now we have $X_f(\F_q) = U(\F_q) \cup W(\F_q)$. However, counting $\F_q$-points we see
    \[\#X_f(\F_q) \leq \#U(\F_q) + \#W(\F_q) \leq 2 \frac{q^{s+1}-1}{q-1} < X_f(\F_q),\]
    by Proposition \ref{prop:reduced_facts}(iii), yielding another contradiction. Thus there exist $v,w \in X_f(\F_q) - U(\F_q)$ satisfying the desired hypotheses.
\end{proof}

\begin{proposition}
\label{prop:leep_and_yeomanry}
	Let $d = 7$ and $n \geq 49$. Suppose \qbound{591} and $F \in \O_K[x_0, \ldots, x_n]$ is a reduced form of degree $7$. Then either $X_f(\F_q)^{\sm} \neq \emptyset$ or there exist $y,z \in X_f(\F_q)$ such that $\Span{y,z} \not\subset X_f$ and $X_f \cap \Span{y,z}$ contains a geometric point of multiplicity 1.
\end{proposition}

\begin{proof}
    Our approach is to find noncolinear points $u,v,w \in X_f(\F_q)$ such that $f|_{\Span{u,v,w}}$ is nonzero and has no linear factors over $\F_q$. Then we can apply Lemma \ref{lem:plane_with_3pts}.	Let $U \subset X_f$ be a linear subvariety defined over $\F_q$ of maximal dimension, denoted by $s$. We have $s \geq 0$ by Proposition \ref{prop:reduced_facts}(iii). We treat the $s=0$, $s=1$, and $s \geq 2$ cases separately.

    \bigskip\noindent\textit{Case $s=0$}. If $s=0$, then by Proposition \ref{prop:reduced_facts}(iii) we have $\#X_f(\F_q) \geq q+1$, so there exist at least four $\F_q$-points on $X_f$. If any four lie on a line, then $X_f(\F_q)^\sm \neq \emptyset$ by Lemma \ref{lem:too_many_linears_d=7}. Otherwise, there exist three points $u,v,w \in X_f(\F_q)$ such that $\Span{u,v,w} \simeq \P^2$. Applying Lemma \ref{lem:plane_with_3pts} to $f|_{\Span{u,v,w}}$, we have the result.

    \bigskip\noindent\textit{Case $s=1$}. Let $U \subset X_f$ be a line defined over $\F_q$. By Lemma \ref{lem:skew_points} there exist $v,w \in X_f(\F_q)$ such that the line $\Span{v,w}$ is skew to $U$ and not contained in $X_f$. If for any point $u \in U(\F_q)$ we have that $f$ vanishes on no lines in the plane $\Span{u,v,w}$, then we may apply Lemma \ref{lem:plane_with_3pts}, as in the $s=0$ case. 

    Assume to the contrary that for all $u \in U(\F_q)$ we have that $f$ vanishes on some line in $\Span{u,v,w}$. By Lemma \ref{lem:7.3}, either $X_f(\F_q)^{\sm} \neq \emptyset$ or there are at most two points $u \in U(\F_q)$ such that $\Span{u,v} \subset X_f$, since $\Span{U,v} \not\subset X_f$ in this case as $s=1$. Similarly, there are at most two $u \in U(\F_q)$ for which $\Span{u,w} \subset X_f$, and if there is a third point $x \in (\Span{v,w} \cap X_f)(\F_q)$, at most two $u \in U(\F_q)$ for which $\Span{u,x} \subset X_f$.

    Thus for all but at most 6 points $u \in U(\F_q)$, we have that $f$ does not vanish on any lines between $u$ and a point on $(X_f \cap \Span{v,w})(\F_q) \supset \{v,w\}$. Instead, $f$ must vanish on a line containing either $v$, $w$, or the potential third point $x$ on $(X_f \cap \Span{v,w})(\F_q)$. When \qbound{11}, this guarantees the existence of at least three points $u_1, \ldots, u_3 \in U(\F_q)$ such that $f$ vanishes on a line $V_i \subset \Span{u_i, v, w}$ through a common point $v$, $w$, or $x$. Without loss of generality, assume $v \in V_i$ for $i=1, \ldots, 3$. Note that $V_i \cap U = \emptyset$; if not we find that $w \in \Span{U,v}$, contradicting our choice of $\Span{v,w}$ skew to $U$.

    Suppose the lines $V_1,V_2,V_3$ are coplanar. Recalling that $V_i \cap U = \emptyset$, we have that the plane $\Span{V_1,V_2,V_3}$ meets $U$ at a point not contained in the union of $V_1$, $V_2$, and $V_3$. By Lemma \ref{lem:too_many_linears_d=7}(ii), $X_f(\F_q)^{\sm} \neq \emptyset$.

    Suppose now that $V_1,V_2,V_3$ are not coplanar, hence $\dim \Span{V_1,V_2,V_3} = 3$. Thus they satisfy the hypotheses of Proposition \ref{prop:codim3bert_d=7}, and we reach the desired conclusion.
        
    \bigskip\noindent\textit{Case $s\geq 2$}. By Lemma \ref{lem:skew_points}, there exist $v,w \in X_f(\F_q)$ such that $\Span{v,w} \not\subset X_f$ and is skew to $U$. First, we argue there exists a chain of linear subvarieties defined over $\F_q$,
    \[u = U^{(s)} \subset U^{(s-1)} \subset \ldots \subset U^{(2)} \subset U^{(1)} \subset U^{(0)} = U,\]
    satisfying $\dim U^{(j)} = s-j$ and
    \begin{equation}\label{eq:U_doesnt_pair}
        \Span{U^{(j)}, x} \not\subset X_f \text{ for all } x \in (X_f \cap \Span{v,w})(\F_q).
    \end{equation}
    Here the superscript serves to remind us that $U^{(j)} \subset U$ has codimension $j$. To find $U^{(1)}$, we recognize that $\Span{U,x} \not\subset X_f$ for any $x \in (X_f - U)(\F_q)$, since $\dim U$ is maximal, so applying Lemma \ref{lem:7.3} there are at most two possible codimension 1 $U^{(1)} \subset U$ defined over $\F_q$ for which $\Span{U^{(1)}, x} \subset X_f$ for each $x \in (X_f \cap \Span{v,w})(\F_q)$. Since there are at most three $x \in (X_f \cap \Span{v,w})(\F_q)$ by Lemma \ref{lem:too_many_linears_d=7}, and more than 6 choices of $U^{(1)}$ by our hypotheses on $q$, there exists at least one satisfying \eqref{eq:U_doesnt_pair}. Repeating this argument, we find $U^{(j+1)} \subset U^{(j)}$ satisfying \eqref{eq:U_doesnt_pair}.

    Suppose now that in $\Span{U^{(1)}, v,w} \simeq \P^{s+1}$, $f$ vanishes on no linear subvarieties of dimension $s$. If $s=1$, we have reduced to the $s=0$ case above, so we are done. If $s > 1$, we may replace $U$ by $U^{(1)}$ and repeat the argument of this case.
    
    Suppose to the contrary that for \textit{all} of the $q+1$ codimension 1 $U_i^{(1)} \subset U$ defined over $\F_q$ satisfying $U^{(2)} \subset U_i^{(1)}$ (which implies \eqref{eq:U_doesnt_pair} for $U_i^{(1)}$), we have that $f|_{\Span{U_i^{(1)}, v, w}}$ vanishes on a linear subvariety $V_i$ of dimension $s$. Each $V_i$ meets the line $\Span{v,w}$ at exactly one of the at most three $\F_q$-points of $X_f \cap \Span{v,w}$. Thus whenever \qbound{18}, there exist five such $V_i$ containing the same $\F_q$-point of $X_f \cap \Span{v,w}$. Without loss of generality, assume $v \in V_1, \ldots, V_5$. 

    Since $\dim (V_i \cap \Span{u,v,w}) =1$, we must have that $f$ vanishes on a line containing $v$ in $\Span{u,v,w}$. This line must meet $\Span{u,w}$, but by hypothesis $\Span{u,v}, \Span{v,w} \not\subset X_f$, so there must exist $y \in (X_f \cap \Span{u,w})(\F_q)$ such that $\Span{v,y} \subset V_i$ for all $i$. Note also that the line $\Span{v,y} \cap U = \emptyset$. 

    If $s > 2$, let $V_i^{(3)} = V_i \cap U^{(2)}$, which has dimension $s-3$.\footnote{Here again the superscript serves to remind us that the codimension of $V_i^{(3)} \subset V_i$ is 3.} Since $f$ vanishes on the span with $v$ of at most two codimension 1 linear subvarieties of $U^{(2)}$ by Lemma \ref{lem:7.3}, at least three of the $V_i^{(3)}$ must coincide. Assume that for $V_1, \ldots, V_3$ we have $V_i^{(3)} = V^{(3)}$. Now we have that $V^{(1)} = \Span{V^{(3)}, v, y} \subset V_i$ for $i=1, \ldots, 3$, so in fact the three $V_i$ meet in codimension 1. If $s=2$, then $V^{(3)} = \emptyset$ and $V^{(1)} = \Span{v,y} \subset V_i$ for all $i$, so we reach the same conclusion.

    Suppose $\dim \Span{V_1,V_2, V_3} = s+1$. Then since $U \cap \Span{V_1,V_2,V_3} \neq \emptyset$, by Lemma \ref{lem:too_many_linears_d=7}(ii) we have $X_f(\F_q)^\sm \neq \emptyset$. Thus we must have $\dim\Span{V_1,V_2,V_3} = s+2$, and we may apply Proposition \ref{prop:codim3bert_d=7}.
\end{proof}

A similar approach gives a degree 11 analogue of Proposition \ref{prop:leep_and_yeomanry}. 

\begin{proposition}
\label{prop:leep_and_yeomanry_d=11}
	Let $d = 11$ and $n \geq 121$. Suppose \qbound{7061} and $F \in \O_K[x_0, \ldots, x_n]$ is a reduced form of degree $11$. Then either $X_f(\F_q)^{\sm} \neq \emptyset$ or there exist $y,z \in X_f(\F_q)$ such that $\Span{y,z} \not\subset X_f$ and $X_f \cap \Span{y,z}$ contains a geometric point of multiplicity 1.
\end{proposition}

\begin{proof}
    The proof follows the same lines as that of Proposition \ref{prop:leep_and_yeomanry}, using the intermediates from \S \ref{subsec:d=11} in place of their analogues from \S \ref{subsec:d=7}.
    
    The $s=0$ case follows from applying Proposition \ref{prop:reduced_facts}(iii) and Lemma \ref{lem:too_many_linears_d=11} to show there exist noncolinear $u,v,w \in X_f(\F_q)$, then invoking Lemma \ref{lem:plane_with_3pts_d=11}. We sketch the $s=1$ case below. We omit a justification of the $s \geq 2$ case, as it follows the same lines as the $s \geq 2$ of Proposition \ref{prop:leep_and_yeomanry} with the modifications described herein.

    Suppose $U \subset X_f$ is a line defined over $\F_q$ and $X_f$ contains no linear subvarieties of dimension 2 defined over $\F_q$. By Lemma \ref{lem:skew_points}, there exist $v,w \in X_f(\F_q)$ such that $\Span{v,w} \not\subset X_f$ and $\Span{v,w} \cap U = \emptyset$. As in the proof of Proposition \ref{prop:leep_and_yeomanry}, if there exists $u \in U(\F_q)$ for which $\Span{u,v,w} \cap X_f$ contains no lines defined over $F_q$, we apply Lemma \ref{lem:plane_with_3pts_d=11} to obtain the desired result.

    Assume instead for all $u \in U(\F_q)$ there exists a line in  $X_f \cap \Span{u,v,w}$ defined over $\F_q$. By Lemma \ref{lem:too_many_linears_d=11}(i), there are at most 20 such $u$ for which $u$ itself is contained in a line in $\Span{u,v,w} \cap X_f$. Since \qbound{40}, there exist five points $u_1,\ldots,u_5$ such that the associated lines $V_1,\ldots,V_5$ all contain a common point. Note also that they are all skew to $U$.
    
    Suppose $v \in V_1,\ldots,V_5$ without loss of generality. If these five lines are coplanar, their span intersects $U$ at a point not contained in any $V_i$. Thus by Lemma \ref{lem:too_many_linears_d=11}(ii) we have $X_f(F_q)^\sm = \emptyset$, since $f$ vanishes on no planes. If the five lines are not coplanar, any three of them which are not coplanar satisfy the hypotheses of Proposition \ref{prop:codim3bert_d=11}, giving the desired conclusion.
\end{proof}

\section{Proof of Theorem \ref{thm:deg711}}\label{sec:proof of thmC}

\begin{proof}[Proof of Theorem \ref{thm:deg711}]
	Let $F \in \O_K[x_0, \ldots, x_n]$ be a reduced form and denote by $f \in \F_q[x_0, \ldots, x_n]$ its image under the residue map. If $X_f(\F_q)^{\sm} \neq \emptyset$ then $X_F(K) \neq \emptyset$ by Hensel's lemma. Assume $X_f(\F_q)^{\sm} = \emptyset$. 
	
	\bigskip
	\noindent\textit{Reducible case.} Consider first that $f$ is geometrically reducible and write $f = \prod_{1 \leq i \leq t} f_i^{e_i}$ over $\overline{\F_q}$. Since neither 7 nor 11 can be written as the sum of composites and $f$ has no geometric linear factors by Proposition \ref{prop:reduced_facts}(ii), at least one of the factors $f_i$ is defined over $\F_q$ with $e_i = 1$.
	
	By Theorem \ref{thm:effBertini_existence}(i), there exists $(\bv,\bw,\bz)$ such that $\fvwz$ corresponds to a form with the same factorization type as $f$ (see Example \ref{ex:deg7_reducible}). Rehomogenizing, this corresponds to a plane curve $C \subset X_f$ with a geometrically integral component $C'$ of multiplicity 1 defined over $\F_q$ of degree $d'$ strictly less than 7 (resp.\ 11). Examining the possible factorizations of $f$, we always have such a component of degree $d' \leq 3$ when $d=7$ (resp.\ $d' \leq 7$ when $d=11$).
	
	Applying Lemma \ref{lem:finding_sm_pt} with $d'\leq 3$ (resp.\ $d' \leq 7$) we find that \qbound{19} suffices (resp.\ \qbound{942}) to guarantee $\#C(\F_q)^\sm > 0$. This is more than ensured by the hypotheses on $q$, so we conclude whenever $F$ is a reduced form with $f$ geometrically reducible, we have $X_F(K) \neq \emptyset$.
	
	\bigskip
	\noindent\textit{Irreducible case.} Assume $f$ is geometrically irreducible. By Proposition \ref{prop:leep_and_yeomanry}, we find $y,z \in X_f(\F_q)$ such that $\Span{y,z} \not\subset X_f$ and $f|_{\Span{y,z}}$ vanishes on a point of multiplicity 1, necessarily defined over $\overline{\F_q}$ but not $\F_q$.
	
	Possibly after a change of coordinates in $\GL_{n+1}(\O_K)$, we may assume the $\overline{\F_q}$-points of $X_f$ along the line $\Span{y,z}$ are in the $x_0 \neq 0$ affine patch. This allows us to find $(\bv, \bw) \in \F_q^{2n-1}$ such that 
    \[f_{\bv, \bw, \bz}(X,0) = f(1, X +\v_1,\w_2X +\v_2, \ldots,\w_nX +\v_n)\] 
    vanishes on the same points as $f|_{\Span{y,z}}$. In particular, it vanishes on two distinct (singular) $\F_q$-points and has a (geometric) simple root $\alpha_0$ corresponding to the marked (geometric) point on $X_f$ along $\Span{y,z}$. We now treat the degree 7 and 11 cases separately.
    
    \bigskip
    \noindent\textit{(i) $d=7$.} By Theorem \ref{thm:effBertini_modified}, there is a polynomial $\Psi_3^{(\alpha_0)}$ of degree at most 165 for which $f_{\bv,\bw,\bz}(X,Y)$ has no nonconstant factor of degree at most 3 over $\overline{\F_q}$ vanishing on $(\alpha_0,0)$ whenever $\Psi_3^{(\alpha_0)}(\bz) \neq 0$. By the hypothesis on $q$, one such $\bz$ must exist.
    
    Rehomogenizing, let $C \subset X_f$ denote the plane curve cut out by $f_{\bv,\bw,\bz}$ for such a choice of $\bz$. Let $C' \subset C$ denote the geometrically integral component containing the point corresponding to $(\alpha_0,0)$; we have $\deg C' \in \{4,5,7\}$,\footnote{If $C$ has a factor of degree 6 then it also has a linear factor over $\F_q$, so we find a smooth point.} so it must appear with multiplicity 1 and be defined over $\F_q$. If $\deg C' \neq 7$, then we apply Lemma \ref{lem:finding_sm_pt} to find that \qbound{159} suffices to ensure $C(\F_q)^\sm \neq \emptyset$ (the limiting case here is $C'$ a smooth quintic curve).    
    
	It remains to consider the case of $C' = C$ an irreducible degree 7 curve with at least two $\F_q$-points coming from $y$ and $z$. If either of these points are smooth, we are done. If they are both singular, then the geometric genus of $C$ is at most 13 (a smooth plane degree 7 plane curve has genus 15). Again applying \eqref{eq:ns_pt_bound}, with $d'=7$ and $g' \leq 13$, we find $\#C(\F_q)^{\sm} > 0$ whenever \qbound{679}.
    
    Under the hypotheses \qbound{679} and $F$ reduced, we have $X_f(\F_q)^{\sm} \neq \emptyset$, therefore $X_F(K) \neq \emptyset$. It is enough to check this only on reduced forms $F$ by Proposition \ref{prop:reduced_facts}(i), so we conclude $K$ is $C_2(7)$ when \qbound{679}.
        
    \bigskip
    \noindent\textit{(ii) $d=11$.} We follow the same approach. Theorem \ref{thm:effBertini_modified} gives $\Psi_{5}^{(\alpha_0)}$ of degree at most 945 for which $f_{\bv,\bw,\bz}(X,Y)$ has no nonconstant factor of degree at most 5 over $\overline{\F_q}$ vanishing on $(\alpha_0,0)$ whenever $\Psi_5^{(\alpha_0)}(\bz) \neq 0$. 
    
    As above, the assumptions on $q$ guarantee existence of one such $\bz$, and $f_{\bv,\bw,\bz}$ corresponds to a plane curve $C \subset X_f$ with at least two $\F_q$-points and a geometrically integral component $C'$ of degree $d' \in \{6,7, 8, 9, 11\}$. If $d' \neq 11$, \qbound{3153} suffices to guarantee a smooth point (the limiting case is $C'$ a smooth curve of degree 9).
    
    If $d'=11$ we have $g' \leq 43$ and applying \eqref{eq:ns_pt_bound} gives $\#C(\F_q)^\sm > 0$ whenever \qbound{7393}. Therefore $X_F(K) \neq \emptyset$ for all reduced $F$, and $K$ is $C_2(11)$ whenever \qbound{7393}.
\end{proof}

\bibliography{loc_sol}
\bibliographystyle{alpha}

\end{document}